\documentclass[12pt,a4paper,oneside]{amsart}

\usepackage{amsfonts, amsmath, amssymb, amsthm, amscd}
\usepackage{hyperref}
\hypersetup{
  colorlinks   = true, 
  urlcolor     = blue, 
  linkcolor    = blue, 
  citecolor   = red 
}
\usepackage{anysize}
\marginsize{2cm}{2cm}{1.5cm}{1.5cm}
\usepackage[pdftex]{graphicx}

\newtheorem{theorem}{Theorem}[section]
\newtheorem{lemma}[theorem]{Lemma}
\newtheorem{claim}[theorem]{Claim}

\theoremstyle{definition}
\newtheorem{definition}[theorem]{Definition}

\theoremstyle{remark}
\newtheorem{remark}[theorem]{Remark}

\numberwithin{equation}{section}

\DeclareMathOperator{\area}{area}
\DeclareMathOperator{\diam}{diam}
\DeclareMathOperator{\dist}{dist}

\DeclareMathOperator{\inte}{int}
\DeclareMathOperator{\conv}{conv}

\DeclareMathOperator{\lk}{lk}
\DeclareMathOperator{\sgn}{sgn}

\renewcommand{\epsilon}{\varepsilon}

\renewcommand{\phi}{\varphi}

\newcounter{fig}

\title[Convex fair partitions into an arbitrary number of pieces]{Convex fair partitions into an arbitrary number of pieces}

\author{Arseniy~Akopyan{$^\spadesuit$}}

\author{Sergey Avvakumov{$^\clubsuit$}}

\author{Roman~Karasev{$^\diamondsuit$}}

\thanks{{$^\spadesuit$} This work was partially completed while A.~Akopyan was at IST Austria, supported by the European Research Council (ERC) under the European Union's Horizon 2020 research and innovation programme (grant agreement No 716117)}

\thanks{{$^\clubsuit$} Supported by the Austrian Science Fund (FWF), Project P31312-N35, and the European
Research Council under the European Union's Seventh Framework Programme ERC Grant agreement ERC
StG 716424 -- CASe}

\thanks{{$^\diamondsuit$} The research of R.~Karasev was carried out within the state assignment 1.1.1-0029/25 of Ministry of Science and Higher Education of the Russian Federation for IITP RAS}

\address{Arseniy Akopyan, FORA Capital, LLC, Miami, USA}
\email{akopjan@gmail.com}

\address{Sergey~Avvakumov, School of Mathematical Sciences, Tel Aviv University, Tel Aviv 69978, Israel}
\email{savvakumov@gmail.com}

\address{Roman~Karasev, Institute for Information Transmission Problems RAS, Bolshoy Karetny per. 19, Moscow, Russia 127994}
\email{r\_n\_karasev@mail.ru}
\urladdr{http://www.rkarasev.ru/en/}


\begin{document}

\begin{abstract}
We prove that any convex body in the plane can be partitioned into $m$ convex parts of equal areas and perimeters for any integer $m\ge 2$; this result was previously known for prime powers $m=p^k$. We also discuss possible higher-dimensional generalizations and difficulties of extending our technique to equalizing more than one non-additive function.
\end{abstract}

\maketitle

\section{Introduction}
	
In \cite{nandakumar2008} a very natural problem was posed: Given a positive integer $m$ and a convex body $K$ in the plane, cut it into $m$	convex pieces of equal areas and perimeters.

The case $m=2$ of the problem is done with a simple continuity argument. The case $m=2^k$ could be done similarly using the Borsuk--Ulam-type lemma by Gromov~\cite{gromov2003} (see also \cite[Theorems~2.1, 3.1]{klartag2016}), which was used to prove another result, the waist theorem for the Gaussian measure (and the sphere). In \cite{barblsz2010} the case $m=3$ was done.

Further cases, $m=p^k$ for a prime $p$, were established in \cite{ahk2014} and \cite{bz2014} independently (and a~similar but weaker fact was established in \cite{bespamyatnikh2000,soberon2012}). In both papers higher-dimensional analogues of the problem were stated and proved. This time we establish a new series of results:

\begin{theorem}
\label{theorem:main}
Any convex body $K\subset\mathbb R^2$ can be partitioned into~$m$ convex parts of equal area and perimeter, for any integer $m\ge 2$.
\end{theorem}

As in the previous work \cite{ahk2014} and \cite{bz2014}, a ``perimeter'' here may mean any continuous function of a convex body in the plane. More precisely, this real-valued function must be defined on convex bodies (convex compacta with non-empty interior) continuously in the Hausdorff metric; in particular, in the course of our proof we never apply this function to a convex body with empty interior.

An ``area'' may be measured with any finite Borel measure with non-negative density (that is absolutely continuous with respect to the Lebesgue measure) in $K$; for a positive density the proof passes literally and the non-negative density is obtained with the standard compactness argument. 

Compared to the previous work on this and similar problems, this time we have found a way to go beyond the usual equivariant (co)homological argument that restricts the possible result to the prime power case. Our proof builds a solution recursively. To prove its validity we argue by induction and use a certain separation lemma that allows us to use the standard homological arguments modulo different primes at different stages of the induction.

While this paper was reviewed in a journal, we (two of the three current authors) have published a solution to another problem \cite{avvakumov2020} that may be viewed as a properly corrected case $d=1$ of Theorem~\ref{theorem:main-d}. The argument basically follows the lines of the argument in this paper with some modifications. The reader may compare this exposition to \cite{avvakumov2020}.

\subsection*{Organization of the paper}
In Section~\ref{subsection:cstm} we explain the standard scheme which gives a proof of Theorem~\ref{theorem:main} for $m$
a prime power. In Section~\ref{subsection:common} we recall some common facts about the transversality of equivariant maps of polyhedra. The reader who is comfortable with these notions may skip this section. In Section~\ref{subsection:pk} we list the properties of a certain polyhedron $P_p$ (constructed in \cite{bz2014}) which is going to be used in the subsequent proofs; we also show how this polyhedron is important for the Borsuk--Ulam type result needed to complete the proof scheme in Section~\ref{subsection:cstm}.

In Section~\ref{section:2pk} we give a proof of Theorem~\ref{theorem:main} for $m$ twice a prime power. This proof is shorter and more visual than the proof of the general case, but the subsequent proof of the general case does not rely on Section~\ref{section:2pk} and proves this particular case independently.

In Section~\ref{section:arbitrary} we state Lemma~\ref{lemma:function-to-function} and deduce Theorem~\ref{theorem:main} for all $m$ from it. In Section~\ref{section:lemma-proof} we prove the lemma, completing the proof of Theorem~\ref{theorem:main}.

In Appendix~\ref{section:bz-explanations} we briefly outline the construction of the polyhedron $P_p$ in \cite{bz2014}. In Appendix~\ref{section:alt-explanations} we provide an implicit proof of why there must exist a polyhedron $Q_p$ suitable for our purposes in place of $P_p$ independent of results of \cite{bz2014}.

In Appendices \ref{section:higher-pos} and \ref{section:higher-neg} we present a higher-dimensional result that does not fully generalize the two-dimensional case, and an explanation of the difficulties of applying our tools to the true higher-dimensional generalization of the two-dimensional problem when $m$ is not a prime power.

\subsection*{Acknowledgments} We thank Alfredo Hubard and Sergey Melikhov for helpful discussions on the problem, Pavle Blagojevi\'c, Peter Landweber, and the anonymous referees for numerous remarks that helped us improve the text.
	
\section{How the proof for $m=p^k$ works}
\label{section:bz}

\subsection{The configuration space --- test map scheme and the proof of the $m=p^k$ case}
\label{subsection:cstm}

Let $F_m(\mathbb R^2)$ be the configuration space of $m$-tuples $(x_1,\ldots, x_m)$ of pairwise distinct points in the plane. To every such $m$-tuple we uniquely associate (following~\cite{aurenhammer1998}) the weighted Voronoi partition of the plane,
\[
\mathbb R^2 = V_1\cup \dots \cup V_m,
\]
with centers at $x_1,\ldots,x_m$ such that the areas of the intersections $V_i\cap K$ are all equal. This can be done continuously in the configuration $F_m(\mathbb R^2)$ (see~\cite[Section~2]{ahk2014} and Theorem~\ref{theorem:closed-graph} below). Then we produce the map ($\phi$ is for perimeter here)
\[
F_m(\mathbb R^2) \to \mathbb R^m,\quad (x_1,\ldots,x_m) \mapsto \left( \phi(V_1\cap K),\ldots, \phi(V_m\cap K) \right),
\]
and then compose it with the linear quotient by the diagonal (i.e., taking the quotient we view $\mathbb R^m$ as a linear space and the diagonal as its linear subspace)
\[
\Delta = \{(t,t,\ldots, t) : t\in\mathbb R\} \subset\mathbb R^m
\]
to obtain an $\mathfrak S_m$-equivariant continuous map
\[
\tau : F_m(\mathbb R^2) \to \mathbb R^m/\Delta =: W_m,
\]
where $\mathfrak S_m$ is the group of permutations of $m$ elements. The vector space $W_m$ can be interpreted as the orthogonal complement to $\Delta$ in $\mathbb R^m$ and as a $(m-1)$-dimensional irreducible representation of the permutation group $\mathfrak S_m$. With the natural action of $\mathfrak S_m$ on $F_m(\mathbb R^2)$ this map $\tau$ becomes $\mathfrak S_m$-equivariant.

\begin{claim}
\label{claim:bu}
The equipartition problem is solved if the image of the $\mathfrak S_m$-equivariant map 
\[
\tau : F_m(\mathbb R^2) \to W_m
\]
contains $0$. That is, if $\tau(x)=0$ for some $m$-tuple $x\in F_m(\mathbb R^2)$ then the partition of the plane $\mathbb R^2 = V_1\cup \dots \cup V_m$ corresponding to $x$ yields a partition of $K$, $(V_1\cap K)\cup \dots \cup (V_m\cap K)$, into convex parts of equal areas and equal perimeters.
\end{claim}

The proof of the case $m=p^k$ is then done by applying a Borsuk--Ulam type result (essentially established in \cite{vassiliev1988} and explicitly stated in \cite[Theorem~1.10]{ahk2014}) proving that the image of $\tau$ contains $0$, see Lemma~\ref{lemma:bu}. The lemma is proved by first arguing that there is \emph{some} equivariant map whose preimage of $0$ is non-trivial in a certain homological sense, so-called \emph{test map}; and then deducing that \emph{any} equivariant map must have $0$ in its preimage.

When $m$ is not a prime power, however, the needed Borsuk--Ulam type result fails, as in this case there exist equivariant maps $F_m(\mathbb R^2) \to W_m$ not having $0$ in their image, see \cite{bz2014}.

\subsection{Common facts about the transversality of equivariant maps of polyhedra}
\label{subsection:common}

Here, we introduce several common definitions and statements that will be used in the rest of the paper.
We also recommend the textbooks~\cite{matousek2003using,prasolov2006,rs1972,brs1976} for general notions of PL topology. There are many definitions of a polyhedron, equivalent to each other. For our purposes it is sufficient to consider a polyhedron as a subset of the Euclidean space with certain structure.

\begin{definition}
A \emph{cellular decomposition} of a subset $P\subset \mathbb R^N$ is a finite system $\mathcal C$ of pairwise distinct convex polytopes in $\mathbb R^N$ such that

i) The union of $\mathcal C$ is $P$.

ii) For any $F\in\mathcal C$ all its polytopal faces are in $\mathcal C$.

iii) The intersection $F\cap G$ of any two $F,G\in\mathcal C$ is either empty or in $\mathcal C$ and is a polytopal face of both $F$ and $G$.

The members of $\mathcal C$ are called \emph{faces} of the cellular decomposition. If all faces are simplices then the cellular decomposition is called a \emph{triangulation}.
\end{definition}

\begin{definition}
A \emph{polyhedron} is a subset of $\mathbb R^N$ having a cellular decomposition.
\end{definition}

The same polyhedron $P$ may have different cellular decompositions. Whenever we assume a particular cellular decomposition of a polyhedron in a statement or a proof, we may omit mentioning the decomposition and write ``face of the polyhedron'' instead of ``face of the cellular decomposition'' or ``face of the triangulation''. 

\begin{definition}
A \emph{subpolyhedron} $Q$ of a cellular decomposition $\mathcal C$ of $P$ is the union of faces in some subset $\mathcal D\subseteq \mathcal C$. Without loss of generality we assume $\mathcal D$ closed under inclusion, that is, if $F\in \mathcal D$ then $F'\in \mathcal D$ for every face $F'\subset F$. This way $Q$ is itself a polyhedron.
\end{definition}

\begin{definition}
A cellular decomposition $\mathcal D$ of a polyhedron $P$ is a \emph{subdivision} of another cellular decomposition $\mathcal C$ if every face of $\mathcal C$ is a subpolyhedron with respect to $\mathcal D$.
\end{definition}

Note that if $Q\subset P$ is a subpolyhedron with respect to some decomposition of $P$ then it is also a subpolyhedron with respect to any of its subdivisions.

There is a way to produce a triangulation of a polyhedron from a cellular decomposition using \emph{barycentric subdivision}: Take a point $v(F)$ in the relative interior of every face $F$ of the decomposition and produce new simplicial faces as convex hulls of $v(F_0), \ldots, v(F_k)$ whenever $F_0\subset F_1\subset \dots \subset F_k$. If one chooses the points $v(F)$ as barycenters (centers of mass) of faces then applying barycentric subdivision sufficiently many times one obtains arbitrarily fine subdivision, that is, having all faces of diameter at most $\delta$ for any given $\delta>0$.

\begin{definition}
A map $f : P\to \mathbb R^d$ is \emph{piecewise-linear}, or PL, with respect to a cellular decomposition of $P$ when $f$ is affine on every face. 
\end{definition}

\begin{definition}
A PL map $f : P\to \mathbb R^d$ is \emph{transverse to zero} with respect to the cellular decomposition of $P$ when the affine map $f|_F$ has surjective derivative on every face $F\subseteq P$ such that $f(F)\ni 0$. 
\end{definition}

Although the PL and ``transverse to zero'' properties do depend on the choice of decomposition, we may omit mentioning the decomposition in relation to PL or transverse to zero map, whenever we assume a particular cellular decomposition of a polyhedron in a statement. 

The following properties are straightforward for a map $f : P\to\mathbb R^d$ transverse to zero:
\begin{itemize}
\item
Only faces of $P$ of dimension $d$ and greater can have $0$ in their image. 
\item
A restriction of $f$ to a subpolyhedron of $P$ is also transverse to zero.
\end{itemize}
Moreover, if the image of a PL map $f : P\to \mathbb R^d$ does not contain $0$ then $f$ is transverse to zero by definition. 

Given a topological space we are going to say that a \emph{generic} element of this space has a certain property if the set of elements with this property is open and everywhere dense.

\begin{lemma}
\label{lemma:generic}
For a triangulation of $P$, a generic PL map $f : P\to \mathbb R^d$, affine on faces of this triangulation, is transverse to zero.
\end{lemma}

\begin{proof}
A map affine on faces of the fixed triangulation is fully defined by its values on the vertices. In order to perturb such a map one only has to change its values on the vertices, that is, modify the map inside a space of such maps homeomorphic to $\mathbb R^{vd}$, where $v$ is the number of the vertices. 

Within such variety of maps the image of a face $F\subseteq P$ of dimension less than $d$ generically does not touch zero, that is, the variety of maps $f$ such that $f(F)\ni 0$ is a proper algebraic subset of $\mathbb R^{vd}$, which is closed and has zero measure. For any face $F\subseteq P$ of dimension at least $d$ the set of PL maps $f$ such that $f(F)\ni 0$ may have positive measure, but considering the derivative of the map $f|_F$ (its homogeneous linear part) one observes that the non-surjective derivatives also constitute a proper algebraic subset of $\mathbb R^{vd}$, which is closed and has zero measure.

Since the number of faces is finite, then the variety of maps $f$ that are not transverse to zero is a proper algebraic subset of $\mathbb R^{vd}$, closed with zero measure. 
\end{proof}

We will need extensions of these notions for \emph{$G$-equivariant maps}, where a finite group $G$ acts freely on $P$ and acts linearly on $\mathbb R^d$. The above common observations have the following equivariant extensions.

\begin{definition}
\label{definition:compatible-with}
Assume that a finite group $G$ acts on a polyhedron $P$ freely. A cellular decomposition of $P$ and the action of $G$ are \emph{compatible with each other} if for any $g\neq e\in G$ and any face $F\subseteq P$:
\begin{itemize}
\item $gF$ is a face,
\item $F\overset{g}{\longrightarrow} gF$ is an affine homeomorphism,
\item $gF\cap F=\emptyset$.
\end{itemize}
\end{definition}

The following lemma simplifies the practical work with this definition.\footnote{One may ask why we use such a definition of compatibility. Note that there may be different definitions, for example, a free PL action of $G$ and a decomposition that is invariant under it, weaker than ours as an example of a boundary of a triangle rotated by the group $\mathbb Z/3\mathbb Z$ shows. Or a free action of $G$ naturally endowing $P/G$ with a cellular decomposition, stronger than ours as an example of a crosspolytope $\conv\{\pm e_1,\ldots, \pm e_d\}$ with an involution $x\mapsto -x$ shows. So we use a definition that allows us to work with it.} A barycentric subdivision is called $G$-equivariant if the choice of new vertices satisfies $v(gF)=gv(F)$ for any face $F\subseteq P$ and any $g\in G$.

\begin{lemma}
\label{lemma:equivariant-compatible}
Assume that a cellular decomposition of $P$ and a free action of a finite group $G$ on $P$ satisfy the first two items of Definition~\ref{definition:compatible-with}. Then any $G$-equivariant barycentric subdivision is compatible with the action of $G$.
\end{lemma}
\begin{proof}
A $k$-dimensional face $E$ of the barycentric subdivision corresponds to a chain $F_0\subset F_1\subset\dots \subset F_k$ of faces of the original decomposition. Assume that, for some $g\neq e\in G$, $E$ intersects $gE$, that corresponds to the chain $gF_0\subset gF_1\subset\dots \subset gF_k$. This means that $E$ and $gE$ have a common vertex, which by the definition of the barycentric subdivision means that $F_i=g F_j$ for some $i$ and $j$. Since the dimensions in this equality must match, in fact we have $gF_i=F_i$. Then $g$ fixes a point of $F_i$ by the Brouwer fixed point theorem, contradicting the assumption that the action is free.

We have explained the third requirement of being compatible. The first two follow from $v(gF)=gv(F)$ and the fact that a new face $E$ is a simplex inside $F_k$ in the above notation.
\end{proof}

\begin{lemma}
\label{lemma:equivariant-generic}
Let a triangulation of a polyhedron $P$ be compatible with the action of a finite group $G$ and let $G$ act on $\mathbb R^d$ linearly. Then a generic PL with respect to this triangulation and $G$-equivariant map $f : P\to \mathbb R^d$ is transverse to zero.
\end{lemma}
\begin{proof}
We use an adjusted form of the non-equivariant argument in the proof of Lemma~\ref{lemma:generic}. For any orbit of vertices $\{Gv\}$ of $P$, we may prescribe $f(v)$ and extend this by equivariance $f(gv) = g f(v)$. Hence the space of all possible equivariant maps $f : P\to\mathbb R^d$ affine on faces of the given triangulation of $P$ is just $\mathbb R^{vd/|G|}$, where $v$ is the number of the vertices in the triangulation.

For any face $F\subseteq P$, its vertices belong to different $G$-orbits, and are therefore independent. It follows that the variety of $G$-equivariant $f$ such that $f(F)\ni 0$ and the derivative of $f|_F$ is non-surjective is a proper algebraic subset of $\mathbb R^{vd/|G|}$, closed with zero measure. Since the number of faces is finite, then the variety of $G$-equivariant $f$ that are not transverse to zero is a proper algebraic subset of $\mathbb R^{vd/|G|}$, closed with zero measure. 
\end{proof}

Let us show that continuous equivariant maps can be approximated by PL maps with respect to sufficiently fine triangulations. We use the Euclidean metric on $\mathbb R^d$ and assume that any polyhedron inherits the Euclidean metric from its containing Euclidean space. We only consider compact polyhedra, so it is convenient to approximate maps uniformly.

\begin{lemma}
\label{lemma:equivariant-approximation}
Suppose that a finite group $G$ acts freely on $P$ and acts linearly on $\mathbb R^d$. Let $f : P\to \mathbb R^d$ be a continuous $G$-equivariant map.
Then for any $\varepsilon > 0$ there exists $\delta >0$ such that for any triangulation of $P$ that is finer than $\delta$ and compatible with the action of $G$ there is a $G$-equivariant PL map affine on faces of this triangulation that approximates $f$ with error at most $\varepsilon$ and is transverse to zero.

If $f$ does not have $0$ in its image then any its sufficiently close approximation does not have $0$ in its image. 
\end{lemma}
\begin{proof}
By the uniform continuity of $f$ (that follows from compactness of $P$) we find $\delta>0$ such that $\diam f(S) < \epsilon/3$ for any subset $S\subseteq P$ of diameter $\diam S < \delta$. Assume that we have a compatible with the action of $G$ triangulation of $P$ with all faces of diameter less than $\delta$. Modify $f$ so that it remains the same on the vertices of the triangulation and is extended to each face affinely. This new PL map differs from the original $f$ by less than $2\epsilon/3$. Then generically perturb the obtained map using Lemma~\ref{lemma:generic} to make it transverse to zero. 
 
To prove the second claim we use compactness of $P$ and choose $\varepsilon$ smaller than the distance from $f(P)$ to $0\in\mathbb R^d$.
\end{proof}

\begin{definition}
The \emph{Cartesian product of two polyhedra} $P$ and $Q$ with respective cellular decompositions $\mathcal C$ and $\mathcal D$ is their Cartesian product $P\times Q$ as a subset of the corresponding Euclidean space with its \emph{natural cellular decomposition} into all products $F\times G$,  $F\in\mathcal C$, $G\in\mathcal D$. 
\end{definition}

The natural cellular decomposition may be further subdivided into a triangulation. Minding that $P\subset\mathbb R^N$ and $Q\subset\mathbb R^M$ we may say that the natural projections $P\times Q\to P$ and $P\times Q\to Q$ are affine on the faces of any cellular decomposition of $P\times Q$ including the natural one.

For the case of a cylinder $P\times [a,b]$ or $P\times S^1$ and already triangulated $P$, we often use the \emph{product triangulation}, which is also sometimes called a lexicographic triangulation (it can be defined the same way for the product of any two triangulated polyhedra).

\begin{definition}
Suppose that a triangulation of $P$ is compatible with the action of a finite group $G$. Let $I$ denote the line segment $[a,b]$ or the circle $S^1$ and let $I$ be somehow triangulated that is decomposed into segments. A \emph{product triangulation} of $P\times I$ (with respect to the triangulations of $P$ and $I$) is defined as follows. Order the orbits of the vertices of $P$ in arbitrary way; since the action of $G$ is compatible with the triangulation, this induces a linear order on the vertices of every face. Order also the vertices of $I$, denote both orders by '$<$'. Every cylindrical cell of $P\times I$ then has a form $(v_0,\ldots, v_k)\times [c,d]$, where $v_0<\ldots<v_k$ and $c<d$. Split this cell into the simplices 
\[
(v_0\times c, \ldots, v_i\times c, v_i\times d,\ldots v_k\times d),
\]
where $i$ goes from $0$ to $k$.
\end{definition}

Note that any product triangulation has the following standard properties:

\begin{itemize}
\item
The product triangulation is indeed a triangulation and is compatible with the action of $G$ by $g(x,t)=(gx,t)$.
\item 
If $[c,d]$ is a $1$-simplex in the subdivision of $I$ then $P\times [c,d]$ is a subpolyhedron of $P\times I$, also triangulated as a product.
\item 
The sets $P\times \{a\}$ and $P\times \{b\}$ are subpolyhedra of $P\times [a,b]$ triangulated as in the original triangulation of $P$.
\item 
The projection $P\times [a,b]\to [a,b]\subset \mathbb R^1$ is PL. For an iterated product trinangulation $(P\times [a,b])\times [c,d]$ the projection $(P\times [a,b])\times [c,d]\to [a,b]\times [c,d]\subset \mathbb R^2$ is again a PL map.
\item 
If the starting triangulations of $P$ and $I$ are sufficiently fine then the obtained triangulation of $P\times I$ is also sufficiently fine and may be used, for example, in Lemma~\ref{lemma:equivariant-approximation}.
\item 
If $Q$ is a subpolyhedron of $P$ then the product triangulation of $Q\times I$ is a subpolyhedron of the product triangulation of $P\times I$.
\end{itemize}

Whenever a polyhedron $P$ is acted on by a group $G$ we assume that the action of $G$ on $P\times [a,b]$ is by $g(x,t)=(gx,t)$. The following lemma is an analogue of Thom's transversality theorem.

\begin{lemma}
\label{lemma:thom}
Let $P$ be a polyhedron and let a PL map $f : P\times [a,b]\to \mathbb R^d$ be transverse to zero with respect to some cellular decomposition of $P\times [a,b]$. 
Then for every $t_0\in [a,b]$ with finitely many exceptions the restriction of $f$ to $P\times \{t_0\}$ is transverse to zero, where the cellular decomposition of $P\times \{t_0\}$ is produced by intersecting the faces of $P\times [a,b]$ with $P\times \{t_0\}$.
\end{lemma}

\begin{proof}
The preimage $Z=f^{-1}(0)$ is a polyhedron that has a cellular decomposition into convex polytopal faces $Z\cap F$ over the faces $F$ of the cellular decomposition of $P\times [a,b]$. There may exist faces of this cellular decomposition of $Z$ that project to single point on the segment $[a,b]$, those are the finite set of exceptional values of $t$. 

Consider a non-exceptional value $t_0\in [a,b]$ and a point $(x,t_0)$ such that $f(x,t_0)=0$. Its minimal containing face $F$ of $P\times [a,b]$ has surjective derivative $f|_F$. Hence $Z\cap F$ is a convex polytope of codimension $d$ in $F$ and $(x,t_0)$ is in its relative interior. Since $t_0$ is non-exceptional, we see that $t$ is not constant on $Z\cap F$. 

Let us show that the restriction of $f$ to $P\times \{t_0\}$ is transverse to zero for the cell $F\cap P\times \{t_0\}$. Consider the tangent space $TF$ of $F$ and denote the derivative of $f|_F$ by $D : TF\to \mathbb R^d$. Since the tangent space of $Z\cap F$ is the kernel of $D$, the tangent hyperplane $H\subset TF$ defined by $dt = 0$ does not contain $\ker D$ (otherwise, $t$ would be constant on $Z\cap F$). Take a line $L\subseteq \ker D$ not contained in $H$ and observe that the value of $D$ on any tangent vector $V\in TF$ is also attained in the unique point of intersection $(V+L)\cap H$. Hence $D$ restricted to $H$ is also surjective.
\end{proof}

We will use the notion of a \emph{pseudomanifold}, which resembles a homological cycle modulo an integer $q$. The difference is that the faces of a homological cycle carry arbitrary coefficients along with orientations, while the faces of a pseudomanifold only carry orientations.

\begin{definition}
A polyhedron $P$ with a cellular decomposition is a \emph{pseudomanifold modulo $q$ of dimension $d$} if $\dim P = d$, its faces of dimension $d$ are oriented, and for any face $F$ of dimension $d-1$ the number of faces $F'\supset F$ counted with signs given by comparing the orientations of $F$ and $\partial F'$ is divisible by $q$.
\end{definition}

\begin{definition}
A polyhedron $P$ with a cellular decomposition is a \emph{pseudomanifold modulo $q$ of dimension $d$ relative to} a subpolyhedron $Q\subset P$ if $\dim P = d$, $\dim Q\le d-1$, the faces of $P$ of dimension $d$ are oriented, and for any face $F\subset P$ of dimension $d-1$ not contained in $Q$ the number of faces $F'\supset F$ counted with signs given by comparing the orientations of $F$ and $\partial F'$ is divisible by $q$.
\end{definition}

The following lemma shows that the property of being a pseudomanifold does not depend on choosing a particular cellular decomposition.

\begin{lemma}
\label{lemma:subdivision-pseudomanifold}
If $P$ is a pseudomanifold modulo $q$ of dimension $d$ relative to $Q\subset P$ with respect to some cellular decomposition of $P$ then it is a pseudomanifold modulo $q$ of dimension $d$ relative to $Q\subset P$ with respect to any other cellular decomposition for which $Q$ is a subpolyhedron with a suitable choice of orientation of $d$-faces.
\end{lemma}
\begin{proof}
Having two cellular decompositions $\mathcal C$ and $\mathcal C'$, we may consider also its common subdivision with faces $F\cap F'$, $F\in\mathcal C$, $F'\in\mathcal C'$. Hence it is sufficient to consider the case when one cellular decomposition is a subdivision of the other.

Let $\mathcal C$ be a cellular decomposition of $P$ and $\mathcal D$ be its subdivision. Denote the corresponding sets of $d$-dimensional faces by $\mathcal C_d$ and $\mathcal D_d$.

Any $F\in \mathcal C_d$ is split into faces $F=E_1\cup\ldots\cup E_{N(F)}$ in $\mathcal D_d$. A choice of an orientation of $F$ induces the choice of orientations of all the $E_i$. In the opposite direction, if the $E_i$ have orientations satisfying the definition of a $d$-dimensional pseudomanifold then on any $(d-1)$-dimensional face $W = E_i\cap E_j$ the orientations of $\partial E_i$ and $\partial E_j$ must be opposite. Since one may pass from any $E_i$ to any other $E_j$ within $F$ by a sequence of adjacent $E_i,E_{i_2},\ldots, E_{i_{k-1}}, E_j$ then the orientations of all the $E_i$ must correspond to a single orientation of $F$.

In the sum of boundaries $\partial E_1 + \dots + \partial E_{N(F)}$ the $(d-1)$-faces not contained in $\partial F$ are canceled, and those contained in $\partial F$ provide a subdivision of $\partial F$ with the orientations matching that of $\partial F$. The sum $\sum_{F\in\mathcal C_d} (\partial E_1 + \dots + \partial E_{N(F)})$ is the same as the sum $\sum_{E\in\mathcal D_d} \partial E$, and is a subdivision of the sum $\sum_{F\in\mathcal C_d} \partial F$. Hence $\sum_{E\in\mathcal D_d} \partial E$ is contained in $Q$ after cancellations modulo $q$ if and only if $\sum_{F\in\mathcal C_d} \partial F$ is contained in $Q$ after cancellations modulo $q$.
\end{proof}

If $P$ is a $d$-dimensional pseudomanifold modulo $q$ relative to $Q\subset P$ then the natural cellular decomposition and the product triangulation of the cylinder $P\times [a,b]$ also have a natural choice of orientations on their $(d+1)$-dimensional faces. In a standard way, we write $P\times [a,b]$ for convenience, but orient it as a product $[a,b]\times P$, when every cell $(d+1)$ $F\times I\subseteq P\times [a,b]$ is oriented as a Cartesian product $I\times F$. The equality for elementary chain boundaries $\partial (I\times F) = \partial I\times F - I\times \partial F$ then sum up to the equality
\[
\partial ([a,b]\times P) = \{b\}\times P - \{a\}\times P + [a,b]\times \partial P\mod q,
\]
that is, the cylinder $P\times [a,b]$ is a pseudomanifold modulo $q$ relative to $P\times\{a\}\cup P\times \{b\}\cup [a,b]\times Q$ ($Q\times [a,b]$ is a subpolyhedron in the natural cellular decomposition of $P\times [a,b]$). The same applies to the product triangulation by Lemma~\ref{lemma:subdivision-pseudomanifold}.

Note that in the natural cellular decomposition and in the product triangulation of the product, for any $d$-dimensional face $F\subset P\times\{a\}$ there is a unique $(d+1)$-dimensional face $E$ of $P\times [a,b]$ having $F$ in its boundary and the sign of $F$ in $\partial E$ is $-1$. For a $d$-dimensional face $F\subset P\times\{b\}$ there is a unique $(d+1)$-dimensional face $E$ of $P\times [a,b]$ having $F$ in its boundary and the sign of $F$ in $\partial E$ is $+1$. 

\begin{definition}
\label{definition:point-sign}
Let $P$ be a pseudomanifold modulo $q$ of dimension $d$ relative to $Q\subset P$. Assume that a PL map $f : P\to \mathbb R^d$ is transverse to zero. Then a point $x\in f^{-1}(0)$ has sign $+1$ if the restriction of $f$ to the $d$-face $F\ni x$ preserves the orientation and has sign $-1$ otherwise.
\end{definition}

The following lemma is a straightforward generalization of the argument in \cite[Section~2.2]{matousek2003using}.
It explains how the preimage of zero looks like for an equivariant and transverse to zero map from a $(d+1)$-dimensional cylinder pseudomanifold to $\mathbb R^d$. It turns out to be a $1$-dimensional pseudomanifold relative to the ends of the cylinder. Moreover, the directions of its edges incident to the ends of the cylinder coincide with the signs of the corresponding zeroes of the restriction of the map to the ends in the sense of Definition~\ref{definition:point-sign}. 

\begin{lemma}
\label{lemma:equivariant-homotopy}
For a polyhedron $P$ with a free action of a finite group $G$, let a cellular decomposition of $P\times [a,b]$ be compatible with the action of $G$. Let $P\times [a,b]$ be a pseudomanifold modulo $q$ of dimension $d+1$ relative to the union of its three subpolyhedra $P\times \{a\}$, $P\times \{b\}$, and some other $G$-invariant subpolyhedron $Q$. Let $G$ also act linearly on $\mathbb R^d$. Let the action of $G$ preserve the orientation of the faces of $P\times [a,b]$ and preserve the orientation of $\mathbb R^d$. 

Assume that a $G$-equivariant PL map $h : P\times [a,b]\to \mathbb R^d$ is transverse to zero and $h(Q)\not\ni 0$. Let $f(x) = h(x,0)$ and $g(x)=h(x,1)$. Then

i$)$ The set $h^{-1}(0)$ is a graph, that is, a $1$-dimensional polyhedron.

ii$)$ 
The graph $h^{-1}(0)$ can be oriented so that the homological boundary modulo $q$ of $h^{-1}(0)$ is the $0$-cycle modulo $q$ $g^{-1}(0)-f^{-1}(0)$. \footnote{More precisely, every vertex of $h^{-1}(0)$ not belonging to $P\times\{a,b\}$ has the difference between the number of incoming edges and the number of outgoing edges divisible by $q$. A vertex of $h^{-1}(0)$ in $P\times \{b\}$ is adjacent to one edge and its direction $(+1$ for incoming and $-1$ for outgoing$)$ equals the sign of this vertex in $g^{-1}(0)$ in the sense of Definition~\ref{definition:point-sign}. A vertex of $h^{-1}(0)$ in $P\times \{a\}$ is adjacent to one edge and its direction $(+1$ for incoming and $-1$ for outgoing$)$ is opposite to the sign of this vertex in $f^{-1}(0)$ in the sense of Definition~\ref{definition:point-sign}.}

iii$)$ The graph $h^{-1}(0)$ is $G$-invariant and $G$ preserves its orientation. The point sets $f^{-1}(0)$ and $g^{-1}(0)$ are $G$-invariant and the assignment of signs is also preserved by $G$.
\end{lemma}
\begin{proof}
Set $[a,b]=[0,1]$ after an affine transformation on the line.

(i) This follows from transversality, the edges of $h^{-1}(0)$ are its intersections with $(d+1)$-faces of $P\times [0,1]$ and the vertices of $h^{-1}(0)$ are its intersections with $d$-faces of $P\times [0,1]$.


(ii) Let us orient an edge $e\in h^{-1}(0)$ adjacent to a vertex $v\in h^{-1}(0)$. The edge $e$ is $h^{-1}(0)\cap E$ for a $(d+1)$-dimensional face $E$ of $P\times [0,1]$, and the vertex $v$ is $h^{-1}(0)\cap F$ for a $d$-dimensional face $F\subset E$ of $P\times [0,1]$. Choose affine coordinates $x_0,\ldots, x_d$ on $E$ so that $x_0\le 0$ on $E$ and $x_0=0$ precisely on $F$. Assume that the form $dx_0\wedge\cdots\wedge dx_d$ corresponds to the orientation of $E$, then the form $dx_1\wedge\cdots \wedge dx_d$ corresponds to the orientation of the face $F$ in the boundary $\partial E$.

The orientation of $e$ is determined so that its orientation together with the orientation of its space of normal vectors (in the tangent space of $E$) pulled back from $\mathbb R^d$ by $h$ (that is, the orientation by the coordinates $h_1,\ldots, h_d$ of the map $h$) together produce the orientation of $E$. Then 

\begin{itemize}
\item
$x_0$ increases along the direction of $e$ if and only if the orientations of its space of normal vectors by $x_1,\ldots, x_d$ and by $h_1,\ldots,h_d$ coincide. Then the sign of $v$ in $\partial e$ is $+1$. In this case the orientations of $F$ by $x_1,\ldots,x_d$ and by $h_1,\ldots,h_d$ coincide.

\item
$x_0$ decreases along the direction of $e$ if and only if the orientations of its space of normal vectors by $x_1,\ldots, x_d$ and by $h_1,\ldots,h_d$ differ. Then the sign of $v$ in $\partial e$ is $-1$. In this case the orientations of $F$ by $x_1,\ldots,x_d$ and by $h_1,\ldots,h_d$ differ.
\end{itemize}

It follows that the contribution of $e$ at $v$ (the sign ov $v$ in $\partial e$) equals the sign produced by comparing the orientation of $F\ni v$ by $h_1,\ldots,h_d$ from its orientation as a part of $\partial E$. The former sign only depends on $v$ and to every $E\supset F$ there corresponds an edge $e = h^{-1}(0)\cap E$ with endpoint $v$. It follows that when $v\not\in P\times \{0,1\}$ then the pseudomanifold modulo $q$ definition can be applied because $F$ and $E$ are not contained in $Q$. This implies that the sum of contributions of all edges of $h^{-1}(0)$ at $v$ is divisible by $q$.

The remaining vertices of $h^{-1}(0)$ lie in either $f^{-1}(0)\times\{0\}$ or $g^{-1}(0)\times\{1\}$. From the transversality assumption and the structure of the cylinder, for any such vertex $v$ and its containing $d$-face $F$, there is a unique $(d+1)$-face $E$ containing $F$ and the unique edge $e$ of $h^{-1}(0)$ adjacent to $v$. The sign of $v$ as a point in $f^{-1}(0)\times\{0\}$ or $g^{-1}(0)\times\{1\}$ by Definition~\ref{definition:point-sign} is produced by comparing the orientation pulled back from $\mathbb R^d$ by $f$ or $g$ (that is, by $h$) on the face $F$ and the orientation of $F$ as a face of $P\times [0,1]$ (by Definition~\ref{definition:point-sign}). 
\begin{itemize}
\item 
For $t=1$, one may take $x_0=t-1$ in the above argument and conclude that the orientation of $F$ as a part of $\partial E$ is the same as its original orientation in $P$.
\item
For $t=0$, one may take $x_0=-t$ in the above argument and conclude that the orientation of $F$ as a part of $\partial E$ is the opposite to its original orientation in $P$.
\end{itemize}
It follows that the sign contribution of this unique edge $e$ at its endpoint $v$ equals the sign of $v$ as a point in $g^{-1}(0)\times\{1\}$ or the opposite to the sign of $v$ as a point in $f^{-1}(0)\times\{0\}$.

(iii) The action of $G$ preserves the orientations of $P\times [0,1]$ and $\mathbb R^d$ by the assumption and therefore preserves the orientation of the graph $h^{-1}(0)$, and preserves the signs on $\partial h^{-1}(0) = g^{-1}(0)-f^{-1}(0)$.
\end{proof}

\begin{remark}
\label{remark:truncated-cylinder}
Suppose that $h:P\times [a,b]\to \mathbb R^d$ is transverse to zero with respect to a product triangulation and for some $[t_1, t_2]\subseteq [a,b]$ the restrictions $h|_{P\times\{t_1\}}$ and $h|_{P\times\{t_2\}}$ are also transverse to zero, for example, when $t_1$ and $t_2$ are given by Lemma~\ref{lemma:thom}. Then the restriction of $h$ to $P\times [t_1,t_2]$ is transverse to zero with respect to the cellular decomposition of $P\times [t_1,t_2]$ induced by the triangulation of $P\times [a,b]$ and we can apply the above Lemma~\ref{lemma:equivariant-homotopy} to it.
\end{remark}

Note that our definition of ``transverse to zero'' depends on the choice of a triangulation or a cellular decomposition and may be lost after passing to a subdivision. The following lemma allows us to restore the transversality and count preimages of zero with signs. 

\begin{lemma}
\label{lemma:subdivision-transverse}
Let $P$ be a pseudomanifold modulo $q$ of dimension $d$ relative to $Q\subset P$ with respect to a triangulation $\mathcal T$. Let $\mathcal T$ be compatible with the action of a finite group $G$, and let $Q$ be $G$-invariant, and let $G$ also act linearly on $\mathbb R^d$. Let $f:P\to \mathbb R^d$ be a $G$-equivariant transverse to zero with respect to $\mathcal T$ map such that $f(Q)\not\ni 0$. Let a triangulation $\mathcal S$ be a subdivision of $\mathcal T$.

If a transverse to zero $G$-equivariant $\widetilde f$ is PL with respect to $\mathcal S$ and is uniformly sufficiently close to $f$ then there is a sign and $G$-orbit preserving bijection between the finite sets $f^{-1}(0)$ and $\widetilde f^{-1}(0)$ $($where the signs of $f^{-1}(0)$ are defined with respect to $\mathcal T)$.
\end{lemma}

The proof of Lemma~\ref{lemma:subdivision-transverse} relies on the following more general lemma. Besides the proof of Lemma~\ref{lemma:subdivision-transverse}, the following lemma is also needed in the situation of Lemma~\ref{lemma:bu} below, where the map we have is not transverse to zero in our definition, but is transverse to zero in a certain ``topological sense'', and we need to make it transverse to zero in our definition.

\begin{lemma}
\label{lemma:local-degree}
Let $P$ be a pseudomanifold modulo $q$ of dimension $d$ relative to $Q\subset P$, let its triangulation be compatible with the action of a finite group $G$, and let $Q$ be $G$-invariant. Let $G$ also act linearly on $\mathbb R^d$. Let $f : P\to \mathbb R^d$ be a $G$-equivariant PL map such that $f(Q)\not\ni 0$. Assume that for every point $x$ such that $f(x)=0$ there is a subpolyhedron $U_x\ni x$ in the same triangulation such that $U_x$ contains $x$ in its interior, $f(U_x)$ is a PL ball containing $0$ in its interior, and $f|_{U_x}$ is a homeomorphism. Also assume that the interiors of $U_x$ and $U_y$ are disjoint for $x\neq y$.

Then for all sufficiently small transverse to zero $G$-equivariant perturbations $\widetilde f : P\to \mathbb R^d$ of $f$ PL with respect to the original triangulation the set $\widetilde f^{-1}(0)$ contains precisely one $x$ in each of $U_x$ and no points outside the $U_x$, and the sign of $x\in \widetilde f^{-1}(0)$ equals $1$ or $-1$ depending on whether $f|_{U_x}$ keeps or flips the orientation of $U_x$ compared to the orientation of $f(U_x)$.
\end{lemma}
\begin{proof}
We perturb $f$ as a PL map with respect to the fixed triangulation. For every face $F\subset U_x$ the property ``$f|_F$ is a homeomorphism onto $f(F)$'' is preserved under sufficiently small perturbation. Moreover, the sign of the determinant of derivative of $f|_F$ for $\dim F = d$ is also preserved.

For a pair of faces $F,F'\subset U_x$ the property $f(F)\cap f(F') = f(F\cap F')$ is also preserved under sufficiently small perturbation of $f$. Since there is a finite list of mentioned properties, under a sufficiently small perturbation all of them will be preserved. Then the images $f(F)$ of all faces $F\subseteq U_x$ produce a triangulation of $f(U_x)$ and therefore the property ``$f|_{U_x}$ is a homeomorphism onto $f(U_x)$ for all $x$'' is preserved under sufficiently small perturbations. In particular, after a sufficiently small perturbation $f^{-1}(0)\cap U_x$ is either a single point or an empty set. 

The property that $f(P\setminus \bigcup \inte U_x)\not\ni 0$ is preserved under sufficiently small perturbation, since the distance from the compact $f(P\setminus \bigcup \inte U_x)$ to zero remains positive. Hence $f^{-1}(0)\subseteq \bigcup \inte U_x$ is preserved under sufficiently small perturbations.

The property $\inte f(U_x)\ni 0$ is preserved less trivially. Since $f(U_x)$ is a PL ball of full dimension in $\mathbb R^d$ (before perturbations), its boundary is homeomorphic to a sphere with corresponding orientation. Hence $\partial U_x$ is also homeomorphic to a sphere. Then the degree of the map $s: \partial U_x\to S^{d-1}$ of topological spheres defined by $s(y)=\frac{f(y)}{|f(y)|}$ is originally $\pm 1$ and is preserved for sufficiently small perturbations of $f$, for example, by homotopy invariance of the mapping degree. Then $\inte f(U_x)\not \ni 0$ would imply $f(U_x)\not \ni 0$, which would imply that $s$ extends to $U_x$ continuously and has zero mapping degree, a contradiction. Hence $f(U_x)\ni 0$ and we have shown that for sufficiently small perturbations of $f$ the set $f^{-1}(0)\cap U_x$ is a single point. 

Generically a small $G$-equivariant perturbation $\widetilde f$ of $f$ preserving all the above mentioned properties will be transverse to zero by Lemma~\ref{lemma:equivariant-generic} and will satisfy the requirement in the statement of the lemma.
\end{proof}

\begin{proof}[Proof of Lemma~\ref{lemma:subdivision-transverse}]
For every $x\in f^{-1}(0)$ let $U_x$ be the face of $\mathcal T$ containing $x$. From transversality with respect to $\mathcal T$, $f|_{U_x}$ is a homeomorphism onto $f(U_x)$ and $f(U_x)$ is a simplex containing $0$ in its interior. Lemma~\ref{lemma:local-degree} then applies to perturbations affine on faces of $\mathcal S$.
\end{proof}

The following lemma is also a straightforward generalization of the argument in \cite[Section~2.2]{matousek2003using}. It says that if we count the difference between the number of orbits of positive and negative preimages of zeroes for two different transverse to zero $G$-equivariant maps we get the same number modulo $q$.

\begin{lemma}
\label{lemma:equivariant-euler}
Let $P$ be a triangulated pseudomanifold modulo $q$ of dimension $d$ relative to $Q\subset P$. Let the triangulation of $P$ be compatible with the action of a finite group $G$ and let $Q$ be $G$-invariant. Let $G$ also act linearly on $\mathbb R^d$. Let the action of $G$ preserve the orientation of the faces of $P$ and preserve the orientation of $\mathbb R^d$. Let $f,g : P\to \mathbb R^d$ be $G$-equivariant transverse to zero maps, possibly with respect to different triangulations of $P$.

Assume that the restrictions $f|_Q$ and $g|_Q$ are homotopic by a $G$-equivariant continuous homotopy $h$ such that $h(Q\times [0,1])\not\ni 0$. Then 

i$)$ To every $G$-orbit of points in $f^{-1}(0)$ and $g^{-1}(0)$ there corresponds a well-defined sign as in Definition~\ref{definition:point-sign};

ii$)$ The sum of signs over $G$-orbits in $f^{-1}(0)$ equals the sum of signs over $G$-orbits in $g^{-1}(0)$ modulo $q$.


\end{lemma}
\begin{proof}
If the triangulations with respect to which $f$ and $g$ are transverse to zero are different, or the triangulation with respect to which $Q$ is a subpolyhedron is different from either of those, we pass to a common cellular subdivision consisting of intersections of faces of different triangulations. And then take a barycentric subdivision to obtain a common triangulation. Perturbing the maps $f$ and $g$ in this common triangulation slightly we make them transverse to zero by Lemma~\ref{lemma:equivariant-generic}. By Lemma~\ref{lemma:subdivision-transverse}, this perturbation does not change the number and signs of elements in the sets $f^{-1}(0)$ and $g^{-1}(0)$. Under sufficiently small perturbations the assumptions $f(Q)\not\ni 0$ and $g(Q)\not\ni 0$ are preserved and the restrictions $f|_Q$ and $g|_Q$ remain homotopic by a $G$-equivariant continuous homotopy (denoted by the same letter $h$) such that $h(Q\times [0,1])\not\ni 0$.

Taking an iterated barycentric subdivision of $P$ and also subdividing $[0,1]$ we may get an arbitrary fine product triangulation of $P\times [0,1]$. By Lemma~\ref{lemma:equivariant-approximation} we may approximate $h : Q\times [0,1] \to \mathbb R^d$ arbitrarily well by a map PL with respect to this fine product triangulation of $P\times [0,1]$. This approximation is defined by its values on the vertices of a triangulation and then extended affinely on faces. Since the restriction of $h$ to the subpolyhedron $Q\times\{0,1\}$ was already PL and we may preserve the same values on the vertices in $Q\times \{0,1\}$ for the approximation, we may assume that the approximation coincides with $h$ on $Q\times\{0,1\}$ and therefore coincides with $f$ or $g$ there. By compactness of $Q\times [0,1]$, a sufficiently close approximation also avoids zero on $Q\times [0,1]$. Denoting the approximation by the same letter $h$, we may assume that $h$ is PL with respect to our fine product triangulation of $P\times [0,1]$.

We then $G$-equivariantly extend the map defined by $f,g,h$ on $P\times \{0\}\cup P\times\{1\}\cup Q\times [0,1]$ to a $G$-equivariant homotopy
\[
h : P\times [0,1]\to \mathbb R^d.
\]
This may be done by first defining $h$ on the vertices where it was not defined previously in an arbitrary $G$-equivariant way and then extending it to all remaining simplices affinely.

Note that after the subdivisions the maps $f(x)=h(x,0)$ and $g(x)=h(x,1)$ need not be transverse to zero, and $h$ need not be transverse to zero. By Lemma~\ref{lemma:equivariant-generic}, we may perturb $h$ very slightly to make $h$, $f(x)=h(x,0)$, and $g(x)=h(x,1)$ transverse to zero. By Lemma~\ref{lemma:subdivision-transverse}, this perturbation does not change the number and signs of elements in the sets $f^{-1}(0)$ and $g^{-1}(0)$.


%

We are now able to apply Lemma~\ref{lemma:equivariant-homotopy} to $h$. By Lemma~\ref{lemma:equivariant-homotopy}(i,ii), $h^{-1}(0)$ is an oriented graph, let $V$ be the set of its vertices and $E$ the set of its edges. For $v\in V$ let $\deg v$ be the difference between the number of incoming edges and outgoing edges in $v$. By Lemma~\ref{lemma:equivariant-homotopy}(iii) the orientation of the graph $h^{-1}(0)$ is preserved by the action of $G$ and therefore $\deg v$ is constant on orbits, that is a function of the quotient $V/G$. Consider 
\[
S = \sum_{v\in V/G} \deg v = \frac{1}{|G|} \sum_{v\in V} \deg v.
\]
Each edge $e\in E$ contributes $1$ to its endpoint and $-1$ to its startpoint and therefore $S=0$.

We know that
\[
S_I = \sum_{v\in (V\setminus P\times \{0,1\})/G} \deg v = 0\mod q
\]
from Lemma~\ref{lemma:equivariant-homotopy}(ii), since every such $\deg v$ is divisible by $q$. The remaining part
\[
S - S_I = S_O = \sum_{v\in (V\cap P\times \{0,1\})/G} \deg v = \sum_{v\in (f^{-1}(0)\times\{0\})/G\cup (g^{-1}(0)\times\{1\})/G} \deg v = 0\mod q
\]
as well. By Lemma~\ref{lemma:equivariant-homotopy}(ii) $S_O$ equals the sum of signs over the $G$-orbits in $g^{-1}(0)$ minus the sum of signs over the $G$-orbits in $f^{-1}(0)$, where Lemma~\ref{lemma:equivariant-homotopy}(iii) ensures that the same sign is assigned to the whole orbit of vertices, proving (i). Hence the sum of signs over the $G$-orbits in $g^{-1}(0)$ equals the sum of signs over the $G$-orbits in $f^{-1}(0)$ modulo $q$, proving (ii).
\end{proof}

\begin{remark}
The proofs of Lemmas~\ref{lemma:equivariant-homotopy} and \ref{lemma:equivariant-euler} also work in the case when the action of $G$ changes the orientations of $P$ and $\mathbb R^d$, but only changes both orientations simultaneously. We only use the given version but this slightly more general version may be useful elsewhere. If $G$ does not change the orientations of $P$ and $\mathbb R^d$ then $P/G$ is a pseudomanifold modulo $q$ itself and one may work in $P/G\times [0,1]$ in a simpler fashion counting points, not orbits of points.
\end{remark}

Parts of our main argument deal with curves or homological cycles on two-dimensional surfaces. More precisely, we will use the intersection number $\xi\cdot \eta$ of two $1$-dimensional cycles modulo $q$ on an oriented surface. The argument below is fairly standard, but we include it for completeness. In the proofs in this paper we only need $1$-dimensional PL cycles modulo $q$, that is collections of oriented segments with coefficients, whose homological boundary (the algebraic sum of their endpoints, a $0$-cycle) is divisible by $q$. We call them \emph{$1$-cycles modulo $q$}.

\begin{definition}
Let $S$ be a two-dimensional oriented compact PL surface with boundary, so that its boundary $\partial S$ contains two closed PL subsets $A$ and $B$. Let $\xi$ be a $1$-cycle modulo $q$ in $S\setminus B$ relative to $A$ ($\partial\xi\subseteq A$), $\eta$ be a $1$-cycle modulo $q$ in $S\setminus A$ relative to $B$ ($\partial\eta\subseteq B$). For $\xi$ and $\eta$ transverse to each other, the \emph{intersection number} $\xi\cdot\eta$ is defined by counting modulo $q$ their transversal intersection points with coefficients and signs coming from comparing the orientation of the surface and the orientation corresponding to $v\wedge w$, where $v$ and $w$ are the intersecting oriented edges of $\xi$ and $\eta$ respectively. For arbitrary $\xi$ and $\eta$, it is defined by perturbing them to the PL transverse situation and checking that the thus obtained number does not depend on the perturbation.
\end{definition}

\begin{lemma}
\label{lemma:intersection-mod-p}
Let $S$ be a two-dimensional oriented compact PL surface with boundary, so that its boundary $\partial S$ contains two closed PL subsets $A$ and $B$. Let $\xi$ be a $1$-cycle modulo $q$ in $S\setminus B$ relative to $A$, $\eta$ be a $1$-cycle modulo $q$ in $S\setminus A$ relative to $B$. Then 

i$)$ The intersection number modulo $q$, $\xi\cdot \eta$ is well-defined. 

ii$)$ If $\xi'$ is a $1$-cycle modulo $q$ in $S\setminus B$ relative to $A$ homologous to $\xi$ then $\xi'\cdot \eta = \xi\cdot \eta$ as residues modulo $q$.

iii$)$ If the intersection number $\xi\cdot \eta$ is not divisible by $q$ then $\xi$ and $\eta$ intersect.
\end{lemma}

When reading the proof below it may be helpful to know that in our applications of this essentially known lemma we only use its two cases: a) $S=\mathbb S^1\times [-\infty,\infty]$ (the infinite segment assumed homeomorphic to a finite one), $A = \mathbb S^1\times \{-\infty,\infty\}$, and $B=\emptyset$; b) $S=[-1,1]\times [a,b]$, $A=\{-1,1\}\times [a,b]$, and $B=[-1,1]\times\{a,b\}$. 

\begin{proof}
We start with (i). Assume we have two perturbations of $\xi$: $\xi'$ and $\xi''$; and two perturbations of $\eta$: $\eta'$ and $\eta''$, so that $\xi'$ is transverse to $\eta'$, $\xi''$ is transverse to $\eta''$. Consider yet another pair of perturbations $\xi'''$ and $\eta'''$ so that $\xi'''$ is transverse to $\eta'$,$\eta''$, $\eta'''$; $\eta'''$ is transverse to $\xi'$, $\xi'', \xi'''$, this is possible for generic perturbations.

Since all perturbations of $\xi$ are homologous to it, and all perturbations of $\eta$ are homologous to it, (i) follows by applying (ii) several times:
\[
\xi'\cdot\eta' = \xi'''\cdot\eta' = \xi'''\cdot \eta''' = \xi''\cdot\eta''' = \xi''\cdot\eta''.
\]

For part (ii), it is sufficient to consider an image of a triangle $T$ in $S\setminus B$ such that the algebraic difference between $\xi$ and $\xi'$ is the image of $\partial T$ modulo something belonging to $A$. Since $\eta$ does not touch $A$, it is sufficient to prove that the number of intersections of the image of $\partial T$ and $\eta$ counted with signs is divisible by $q$. After a perturbation, we may assume the map $T\to S$ piece-wise linear and generic keeping its boundary the same and transverse to $\eta$. After that we may further split $T$ into triangles that are mapped to $S$ by an embedding. 

After these reductions we consider an embedded triangle $T\subseteq S$ and the transverse intersection of the cycle $\eta$ and $\partial T$, counting its points modulo $q$ with signs. Consider the intersection $\eta\cap T$, this is a $1$-cycle modulo $q$ in $T$ relative to $\partial T$, since $\eta$ is a $1$-cycle modulo $q$ relative to $B$ and $T$ is disjoint from $B$.

Consider the contribution of an edge of $\eta$ to $\eta\cdot\partial T$. It is $0$ if both ends of the edge are inside $T$ or both outside $T$; otherwise, it is $+1$ if the endpoint is inside $T$ (and the starting point is outside) and $-1$ if the endpoint is outside $T$ (and the starting point is inside). In other words, each edge contributes the sum of two numbers: the first is $+1$ if the endpoint is inside $T$ and $0$ otherwise, the second is $-1$ if the starting point is inside $T$ and $0$ otherwise. Summing this up over all the edges we see that $\eta\cdot\partial T$ is the sum over the vertices of $\eta$ inside $T$ of the incoming edges minus the sum of the outgoing edges. For each vertex of $\eta$ inside $T$ this sum is $0$ modulo $q$ because $\eta$ is a cycle relative to $B$ and $T$ is disjoint with $B$.


In order to establish (iii) assume that $\xi$ and $\eta$ do not intersect. Then they may be perturbed to PL cycles $\xi'$ and $\eta'$ so that $\xi'$ and $\eta'$ do not intersect. Since in this case $\xi'$ and $\eta'$ are transverse to each other, their intersection number must be divisible by $q$, a contradiction.
\end{proof}

\subsection{The polyhedron in the configuration space}
\label{subsection:pk}

In \cite{bz2014} it was shown that there exists a polyhedron $P_m\subset F_m(\mathbb R^2)$ of dimension $m-1$, which is $\mathfrak S_m$-equivariantly homotopy equivalent to the whole $F_m(\mathbb R^2)$. We briefly recall its properties:

\begin{itemize}
\item
$P_m$ has dimension $m-1$, it is invariant with respect to the action of the permutation group $\mathfrak S_m$ on $F_m(\mathbb R^2)$ that affinely maps its faces to faces.
\item
There is a cellular structure on $P_m$ (not a triangulation, cells are not simplices but convex polytopes) with a single orbit of the top-dimensional cells of $P_m$ under the action of $\mathfrak S_m$. The top-dimensional cells of $P_m$ may be oriented so that $\mathfrak S_m$ acts on these orientations by the sign of the permutation \cite[Lemma~4.1]{bz2014}. 
\item
For $m=p^k$, $P_m$ becomes a pseudomanifold modulo $p$ (essentially \cite[Lemma~4.2]{bz2014}, see also the explanations in Section~\ref{section:bz-explanations}) with given orientations of the cells, that is the homological boundary of the $(m-1)$-chain of its top-dimensional cells with given orientation is divisible by $p$. 
\item
There exists an $\mathfrak S_m$-equivariant PL map $f : P_m\to W_m$ (of~\cite[Lemma~4.1]{bz2014}) such that $f^{-1}(0)$ consists of precisely one point in the interior of each top-dimensional cell of $P_m$. At a neighborhood of every point of $f^{-1}(0)$ the sign of $f$ in the sense of Lemma~\ref{lemma:local-degree} is $+1$ relative to the given orientation of $P_m$ and a fixed orientation on $W_m$.
\end{itemize}

\emph{We only use the stated properties of $P_m$ in our proof.} For further details of this construction we refer to the original paper \cite{bz2014}. Also note that our main proof in Sections~\ref{section:arbitrary} and \ref{section:lemma-proof} only uses $P_m$ for $m$ equals a prime $p$. In view of this, in Section~\ref{section:alt-explanations} we outline an alternative implicit construction of the polyhedra $Q_p\subset F_p(\mathbb R^2)$ that can be used instead of $P_p$ in our main proof. 

The following lemma is central in the proof of the equipartition result for $m=p^k$, it provides the facts about the restriction of $\tau : F_m(\mathbb R^2)\to W_m$ from Section~\ref{subsection:cstm} to $P_m$ that we are going to use in our main argument.

\begin{lemma}
\label{lemma:bu}
Let $m$ be a power of a prime $p$, $G_m=\mathfrak S_m$ for $p=2$ or the even permutation subgroup of $\mathfrak S_m$ for odd prime $p$. Then

i$)$ There exists an $\mathfrak S_m$-equivariant PL map $\tau_0 : P_m\to W_m$ that is transverse to zero and has the number of $G_m$-orbits of points in $\tau_0^{-1}(0)$ counted with signs not divisible by $p$. 

ii$)$ Any transverse to zero $G_m$-equivariant PL map $\tau : P_m\to W_m$ has the number of $G_m$-orbits of points in $\tau^{-1}(0)$ counted with signs not divisible by $p$;

iii$)$ Any continuous $G_m$-equivariant map $\tau : P_m\to W_m$ contains $0$ in its image.
\end{lemma}
\begin{proof}
Lemma~\ref{lemma:equivariant-compatible} produces a triangulation of $P_m$ that is compatible with the action of $\mathfrak S_m$. In fact, the construction of $P_m$ outlined in Section~\ref{section:bz-explanations} is a variant of the barycentric subdivision that is compatible with the action of $\mathfrak S_m$.

From~\cite[Lemma~4.1]{bz2014} (see details in Appendix~\ref{section:bz-explanations}) we obtain a $\mathfrak S_m$-equivariant map $f : P_m\to W_m$. This $f$ is not transverse to zero in our definition, but it is a local PL homeomorphism near every point of $f^{-1}(0)$ and has a unique $\mathfrak S_m$-orbit of points in $f^{-1}(0)$, all having sign $+1$ since $\mathfrak S_m$ acts on the orientations of both $P_m$ and $W_m$ by the permutation sign. A transverse to zero $\tau_0 : P_m\to W_m$ is then obtained by a sufficiently small $\mathfrak S_m$-equivariant perturbation of $f$ using Lemma~\ref{lemma:local-degree}.

From the above listed properties, $P_m$ is a pseudomanifold with a free action of $\mathfrak S_m$ and $G_m\subseteq \mathfrak S_m$ preserves its orientation. The point set $\tau_0^{-1}(0)$ (in the description of $P_m$) consists of a single $\mathfrak S_m$-orbit, that corresponds to a single $G_m$-orbit for $m=2$, or a pair of $G_m$-orbits with equal signs (because the odd permutations change the orientations of both $P_m$ and $W_m$ in this case). This finishes the proof of (i).

By Lemma~\ref{lemma:equivariant-euler}, for any other $G_m$-equivariant transverse to zero map $\tau: P_m \to W_m$ the number of $G_m$-orbits in $\tau^{-1}(0)$ counted with the orientation signs is not divisible by $p$. This proves (ii).

For a continuous $G_m$-equivariant map $\tau : P_m\to W_m$ and the proof of (iii) we apply Lemma~\ref{lemma:equivariant-approximation} and find a contradiction if $\tau(P_m)\not\ni 0$.
\end{proof}

\section{Proof for $m=2p^k$}
\label{section:2pk}

Let us start by considering the simplest particular case of the new result. We consider an odd prime $p$ and $m=2p^k$. The proof for arbitrary $m$ (also proving this particular case) given in Sections~\ref{section:arbitrary} and \ref{section:lemma-proof} is more technical. 

Take a parameter $t\in S^1$ in the unit circle which we interpret as an angle, so that $S^1=\{(\cos t, \sin t)\}$. Cut $K$ by a straight oriented line directed along $(\cos t, \sin t)$ into equal area halves, it is uniquely done given the direction $t$. Denote the half of $K$ to the left of the line by $L_t$.

Before going into details, let us explain the rest of the proof in a few sentences. As we know from Section~\ref{subsection:cstm}, for each $t$ the body $L_t$ can be partitioned into $p^k$ convex parts of equal area and equal perimeter, the partition need not be unique. We plot all the common perimeters of the parts that we can achieve in this way as a graph of a (multivalued) function $S^1\to \mathbb R$. This graph is a (closed) set in $S^1\times \mathbb R$ and is denoted $\zeta$ below. The key idea, is that 1) generically  $\zeta$ is a modulo $p$ $1$-cycle and 2) the number of solutions to the problem of partitioning $L_t$ is not divisible by $p$ and so the number of intersections of a generic vertical lines $\{t_0\}\times \mathbb R$ and  $\zeta$ is not divisible by $p$. From that we conclude that $\zeta$ (and even a connected component of $\zeta$) separates the top of the cylinder $S^1\times \mathbb R$ from its bottom, Claim~\ref{claim:separates}, and so $\zeta$ intersects itself after the half rotation by $\pi$, Claim~\ref{claim:cylinder-twodim}, see Figure~\ref{figure:cylender}. The latter means that for some $t$ we can achieve the same value of the perimeter for both $L_t$ (the left half) and $L_{t+\pi}$ (the right half), which solves the problem of partitioning $K$. 
Let us now go into details.

\begin{figure}[ht]
	\label{fig:cylinder}
	\parbox[b]{0.4\textwidth}{
		\begin{center}
			\includegraphics{fig-nandakumar-1}\\
		\end{center}
	}
	\hskip 0.6cm
	\parbox[b]{0.5\textwidth}{
		\begin{center}
			\includegraphics{fig-nandakumar-2}\\
		\end{center}
	}
	\caption{Partitioning into $2p^k$ parts starts with a halving cut shown on the left. On the right we show $S^1\times \mathbb R$ and the common values of the perimeter with $1$-cycles $\zeta$ (continuous line) and $R(\zeta)$ (dashed line).}
\end{figure}

Consider the problem of partitioning $L_t$ into $p^k$ convex parts of equal area and equal perimeter. Apply the scheme of Section~\ref{subsection:cstm} with the configuration space $F_{p^k}(\mathbb R^2)$ replaced by the polyhedron $P_{p^k}\subset F_{p^k}(\mathbb R^2)$ from Section~\ref{subsection:pk}. There appears a $G_{p^k}$-equivariant map 
\[
\tau : P_{p^k}\times S^1\to W_{p^k}.
\]
Let us recall how this map works. A point $(x,t)\in P_{p^k}\times S^1$ defines the partition of $L_t$ into $p^k$ convex parts of equal area that corresponds to $x$. The map $\tau$ first sends $(x,t)$ to the perimeters of these parts, a point in $\mathbb R^{p^k}$, and then projects this point to the orthogonal complement to the diagonal, $W_{p^k}$. The solution set $Z = \tau^{-1}(0)$ then corresponds to partitions of $L_t$ into parts of equal area and equal perimeter.

Perturb $\tau$ slightly using Lemma~\ref{lemma:equivariant-approximation} to make it transverse to zero and PL with respect to a sufficiently fine product triangulation of $P_{p^k}\times S^1$. After the perturbation $Z$ (now denoting the preimage of zero under the perturbed $\tau$) is a one-dimensional pseudomanifold modulo $p$ with a free action of $G_{p^k}$, this is justified by Lemma~\ref{lemma:equivariant-homotopy}(i,ii,iii). The lemma applies to $P_{p^k}\times S^1$ in the same way as to $P_{p^k}\times [0,2\pi]$, while the equality $\tau(x,0)=\tau(x,2\pi)$ allows to conclude that the quotient $Z/G_{p^k}$ is a one-dimensional pseudomanifold modulo $p$, including its vertices corresponding to $t=0, 2\pi$.

On $Z$ we have a function $f:Z\to\mathbb R$ which sends $(z,t)\in Z$ to the average perimeter of the $p^k$ parts of the partition of $L_t$ corresponding to $z$ (since we perturbed $\tau$ slightly, these parts might have slightly different perimeters, all arbitrary close to $f(z,t)$). By definition, $f$ is constant on the orbits of $G_{p^k}$ and so its quotient $\widetilde{f}:Z/G_{p^k}\to \mathbb R$ is well defined. Denote by $\zeta$ the image of $Z/G_{p^k}$ under the map $F(z,t) = (t, \widetilde f(z))$. Although $\zeta$ is a $1$-cycle modulo $p$ in the cylinder $S^1\times \mathbb R$, we denote its support by the same letter.

\begin{claim}
\label{claim:cylinder-twodim}
Let $R$ be the half-rotation $($that is, the rotation by $\pi)$ of the cylinder $S^1\times\mathbb R$.
Then $\zeta\cap R(\zeta)$ is non-empty.
\end{claim}

\begin{proof}[Deduction of the case $m=2p^k$ of the main theorem from Claim~\ref{claim:cylinder-twodim}]
The theorem follows from this claim since a common point $(t,x)$ of $\zeta$ and $R(\zeta)$ corresponds to a pair of partitions of $L_t$ and $L_{t+\pi}$ into $p^k$ parts each such that all the areas in both partitions are equal to $\frac{\area K}{m}$, and all the perimeters in both partitions equal $x$ with arbitrarily fine precision. The standard compactness (of the space of all partitions of $K$ into $m$ parts of equal area by the Blaschke selection theorem~\cite[Section~6.1]{gruber2007}) and going to the limit argument then provides a partition with precisely the same perimeters of the parts.
\end{proof}

To prove Claim~\ref{claim:cylinder-twodim} we first need the following:

\begin{claim}
\label{claim:separates}
A connected component of $\zeta$ separates the bottom $S^1\times \{-\infty\}$ from the top $S^1\times \{+\infty\}$ of the cylinder $S^1\times \mathbb R$.
\end{claim}

\begin{proof}[Proof of Claim~\ref{claim:separates}]
The plan of the proof is as follows:
\begin{itemize}
\item [A)] We choose $t_0\in S^1$ so that in a neighborhood of $P_{p^k}\times\{t_0\}$ the intersection $(Z/G_{p^k})\cap (P_{p^k}/G_{p^k})\times\{t_0\}$ is easy to analyze meaning that it is ``transverse'' in a certain sense. By Lemma~\ref{lemma:bu} we conclude that the algebraic sum of the points in this intersection is not divisible by $p$.
\item [B)] Considering the $F$-image of this intersection $F(Z/G_{p^k})\cap F((P_{p^k}/G_{p^k})\times\{t_0\})=\zeta\cap \{t=t_0\}$ we conclude that the intersection number $\zeta\cdot \{t=t_0\}$ is not divisible by $p$.
\item [C)] From $\zeta\cdot \{t=t_0\}\neq 0$ we deduce the statement of the claim by choosing a suitable connected component of $\zeta$.
\end{itemize}


(A).
Let $[a,b]\subset S^1$ be a $1$-simplex in the subdivision of $S^1$ in the product triangulation $P_{p^k}\times S^1$.
By the construction of the product triangulation, $P_{p^k}\times [a,b]$ is a subpolyhedron of $P_{p^k}\times S^1$, also with the product triangulation. 

Note that we are in the situation when $\tau$ is PL and transverse to zero. By Lemma~\ref{lemma:thom}, for all but a finite number of values $t\in[a,b]$ the restriction of $\tau$ to $P_{p^k}\times \{t\}$, denoted by $\tau_{t}$ for brevity, is transverse to zero. The pseudomanifold $Z$ has a finite number of vertices. So, there are $t_0\in(a,b)$ and a small $\delta > 0$ such that for all values $t\in [t_0-\delta,t_0+\delta]$ the map $\tau_{t}$ is transverse to zero, $(P_{p^k}\times [t_0-\delta, t_0+\delta])\cap Z$ contains no vertex of $Z$ and so consists of several oriented (by Lemma~\ref{lemma:equivariant-homotopy}(ii)) line segments connecting $P_{p^k}\times \{t_0-\delta\}$ to $P_{p^k}\times \{t_0+\delta\}$.

Then every point $(z,t_0)\in \tau_{t_0}^{-1}(0)$ is the intersection of such a line segment with $P_{p^k}\times \{t_0\}$. The sign of the point $(z, t_0)$ in the preimage of zero $\tau_{t_0}^{-1}(0)$ is $+1$ if the coordinate $t$ increases along the line segments and $-1$ otherwise, by Lemma~\ref{lemma:equivariant-homotopy}(ii) applied to the restriction of $\tau$ to $P_{p^k}\times [t_0-\delta, t_0]$ (see Remark~\ref{remark:truncated-cylinder} for the explanation why the lemma applies).

Passing to the quotient by the action of $G_{p^k}$, we see that $(P_{p^k}/G_{p^k}\times [t_0-\delta, t_0+\delta])\cap Z/G_{p^k}$ consists of several oriented line segments $s_i$ connecting $P_{p^k}/G_{p^k}\times \{t_0-\delta\}$ to $P_{p^k}/G_{p^k}\times \{t_0+\delta\}$. 

Since $\tau_{t_0}$ is transverse to zero, by Lemma~\ref{lemma:bu} we know that the sum of signs of $G_{p^k}$-orbits of points in $\tau_{t_0}^{-1}(0)$ is not divisible by $p$. Which means that the number of the segments $s_i$ along which $t$ increases minus the number of the segments $s_i$ along which $t$ decreases is not divisible by $p$.

(B).
Since $F$ preserves the $t$ coordinate, in $1$-cycle $\zeta= F(Z/G_{p^k})$ only the images of the segments $s_i$ intersect the vertical line $\{t=t_0\}$. Moreover, each curve $\sigma_i = F(s_i)$ starts on one side of the vertical line and ends on the other side.

To compute $\zeta\cdot\{t=t_0\}$ we need to consider a PL approximation of $\zeta$ transverse to $\{t=t_0\}$. Such an approximation can be chosen fixed on the endpoints of all $\sigma_i = F(s_i)$. Any PL approximation $\sigma_i'$ of $\sigma_i$ keeping the ends fixed has intersection number $\sigma_i'\cdot \{t=t_0\}$ equal to $-1$ or $+1$ depending on whether the start of $\sigma_i$ is on the right of the vertical line or on the left. The remaining part of $\zeta$, $F(Z/G_{p^k}\setminus\bigcup_i s_i)$, is at the distance at least $\delta$ from the vertical line $\{t=t_0\}$ and does not contribute to $\zeta\cdot\{t=t_0\}$ even after the approximation.
In view of Lemma~\ref{lemma:intersection-mod-p}(ii) we conclude that the intersection number $\zeta\cdot\{t=t_0\}=\sum_i \sigma_i\cdot\{t=t_0\}$ is not divisible by $p$, this is exactly the number of the segments $s_i$ along which $t$ increases minus the number of the $s_i$ along which $t$ decreases.

(C).
Now we explain choosing a connected component of $\zeta$. First, split the graph $Z/G_{p^k}$ into connected components $Y_1\sqcup \dots\sqcup Y_n$. Each connected component is a $1$-cycle modulo $p$. Since $\zeta\cdot\{t=t_0\} = \sum_i \zeta_i\cdot \{t=t_0\}$ is not divisible by $p$, some $\zeta_i\cdot\{t=t_0\}$ is not divisible by $p$.

Now consider the $1$-cycle $\zeta_i=F(Y_i)$ and assume that a curve $\gamma$ passing from the bottom $S^1\times \{-\infty\}$ to the top $S^1\times \{+\infty\}$ of the cylinder $S^1\times\mathbb R$ does not intersect $\zeta_i$. The curve $\gamma$ is homologous to a vertical line relative to the top and the bottom, if we consider the cylinder compactified to $S^1\times [-\infty,+\infty]$. Lemma~\ref{lemma:intersection-mod-p}(ii,iii) then applies and shows that the curve $\gamma$ must also intersect the $1$-cycle $\zeta_i$ since they have intersection number not divisible by $p$. 

The $1$-cycle $\zeta_i$ is connected and its support is contained in a connected component of the support of the $1$-cycle $\zeta$ that we choose as the output of the claim.
\end{proof}

\begin{proof}[Deduction of Claim~\ref{claim:cylinder-twodim} from Claim~\ref{claim:separates}]
From Claim~\ref{claim:separates} we know that a connected component $\zeta_i$ of $\zeta$ splits the cylinder $S^1\times \mathbb R$ into connected parts and separates the top from the bottom. Hence there is unique part $A$ that is infinite at the top and bounded at the bottom. The half-rotated $R(\zeta)$ has the corresponding component of the complement $R(A)$ and $A\cap R(A)\neq\emptyset$. 

Assume that $\zeta_i$ does not intersect $R(\zeta_i)$, in particular $\zeta_i$ does not intersect $\partial R(A)$. Since $\zeta_i$ is connected itself, then either $\zeta_i \cap R(A)=\emptyset$ or $\zeta_i\subset R(A)$. In the first case $A\supset R(A)$, in the second case $A\subset R(A)$, with strict inclusion since their boundaries have no intersection. Those strict inclusions are impossible since by $\pi$-rotating once more they would imply $A\subset R(A)\subset R(R(A))=A$ or $A\supset R(A)\supset R(R(A))=A$ respectively. This is a contradiction proving $\zeta_i\cap R(\zeta_i)\neq\emptyset$.
\end{proof}

\section{Proof for arbitrary $m$}
\label{section:arbitrary}

In our proof of the general case of Theorem \ref{theorem:main}, we are going to use induction, which also resembles the proof of a particular case of the Knaster problem by induction in \cite{yamabe1950}.

It turns out helpful to use the language of multivalued functions on the space of convex bodies $\mathcal K$. In fact, in all our arguments we are going to use the convex bodies contained in the original $K$ with area bounded from below by a sufficiently small positive number $A$. By the Blaschke selection theorem this space $\mathcal K_A$ is a Haudorff compact space in the Hausdorff metric for any $A>0$.

\begin{definition}
A \emph{nice} multivalued function $\mathcal K_A\to (-1,1)$ is determined by its closed graph in $\mathcal K_A\times [-1,1]$, that is given by the equation
\[
\phi(C, y) = 0,
\]
where $\phi: \mathcal K_A\times [-1,1]\to\mathbb R$ is a continuous single-valued function satisfying 
\[
\phi(C, -1) < 0,\quad \phi(C, 1) > 0
\]
for all $C\in \mathcal K_A$.
\end{definition}

By the intermediate value theorem, a nice multivalued function attains at least one value on every $C\in\mathcal K_A$, that is for every $C\in\mathcal K_A$ there exists $y$ such that the pair $(C,y)$ is on the graph of the multivalued function.

Here we restrict the values of a multivalued function to $(-1,1)$, which in practice may be assumed after an appropriate scaling of its values, if the values were in a larger interval $(-L, L)$. Any continuous single-valued function $f : \mathcal K_A\to (-1,1)$ may be considered as a nice multivalued function by putting
\[
\phi(C,y) = f(C) - y.
\] 

We will identify a nice multivalued function on $\mathcal K_A$ with the equation of its graph $\phi(C,y)$, when we need to refer to this function by a name. Theorem \ref{theorem:main} will follow from iterations of the following claim.

\begin{lemma}
\label{lemma:function-to-function}
Assume $\phi$ is a nice multivalued function on $\mathcal K_{A/p}$ and $p$ is a prime. Then there exists another nice multivalued function $\psi$ on $\mathcal K_A$ such that whenever $C\in\mathcal K_A$ satisfies 
\[
\psi(C, y) = 0
\]
then there exists a partition $C = C_1\cup \dots \cup C_p$ into convex bodies of equal area, such that
\begin{equation}
\label{equation:equalized-p}
\phi(C_1, y) = \dots = \phi(C_p, y) = 0.
\end{equation}
\end{lemma}

\begin{proof}[Proof of Theorem \ref{theorem:main} assuming Lemma \ref{lemma:function-to-function}]
Decompose $m$ into primes, $m=p_1p_2\dots p_n$. Let $\phi_1$ be the perimeter single-valued function (or any continuous function of a convex body contained in $K$). Apply the lemma to $\phi_1$ and $p_1$ to obtain $\phi_2$. Then apply the lemma to $\phi_2$ and $p_2$ and so on. The final function $\phi_{n+1}$ will be a nice mutlivalued function of a convex body. 

From the intermediate value theorem, there exists $y\in (-1,1)$ such that 
\[
\phi_{n+1}(C, y) = 0
\]
for the convex body $C$ we are interested in. It means that $C$ may be partitioned into $p_n$ convex bodies of equal area and the same value $y$ of the multivalued function $\phi_n$. Each of these bodies may in turn be partitioned into $p_{n-1}$ parts of equal area and the same value $y$ of the multivalued function $\phi_{n-1}$, and so on. Eventually, we obtain a partition of $C$ into $m=p_1\cdots p_n$ parts of equal area and the same value $y$ of the multivalued function $\phi_1$, which is in fact the single-valued function that we need to equalize.
\end{proof}

\section{Proof of Lemma \ref{lemma:function-to-function}}
\label{section:lemma-proof}

As in Section~\ref{subsection:cstm}, we parametrize some of the partitions of $C\in\mathcal K_A$ into $p$ convex parts of equal area with the polyhedron $P_p$ of \cite{bz2014} (see Section~\ref{subsection:pk} for the list of the properties of $P_p$ that use below). Namely, to each $p$-tuple $(x_1,\ldots,x_p)=:x\in P_p\subset F_p(\mathbb R^2)$ we uniquely associate the weighted Voronoi partition of the plane,
\[
\mathbb R^2 = V_1(x)\cup \dots \cup V_p(x)
\]
with centers at $x_1,\ldots,x_p$ such that the areas of the intersections $C_i(x):=V_i(x)\cap C$ are all equal.

We use the group $G_p\subset\mathfrak S_p$, where $G_2=\mathfrak S_2$ and $G_p$ is the subgroup of even permutation in $\mathfrak S_p$ for odd $p$, as defined in Section~\ref{subsection:pk}. The group $G_p$ acts on $P_p$ freely, since it is a subset of the configuration space of $p$-tuples of pairwise distinct points in $\mathbb R^2$ and $G_p$ permutes those points. The group $G_p$ also acts on $\mathbb R^p$ by permutation of coordinates.

Consider the following $G_p$-equivariant map:

\[
\Phi : \mathcal K_A\times P_p \times [-1,1]\to \mathbb R^p,\quad \Phi(C, x, y) = \left( \phi(C_1(x), y), \phi(C_2(x), y), \ldots, \phi(C_p(x), y) \right).
\] 

\begin{claim}
\label{claim:Phi_is_continuous}
The map $\Phi$ is continuous. 
\end{claim}

While this paper was under review, this continuity claim was independently proved in~\cite[Section~6]{sadovek2023}. For completeness, we provide a proof that is shorter thanks to the following useful theorem.

\begin{theorem}[folklore]
\label{theorem:closed-graph}
Any map $f : X\to Y$ from a Hausdorff compact space $X$ to a Hausdorff compact space $Y$ whose graph is closed is continuous.
\end{theorem}
\begin{proof}
It is sufficient to check that the preimage of any closed $F\subseteq Y$ is closed. The set $X\times F$ is closed in $X\times Y$ and its intersection with the graph $\Gamma_f$ of $f$ is closed. Since $(X\times F)\cap \Gamma_f$ is a closed subset of the compact $X\times Y$ then it is itself compact. Then its projection to $X$ is also compact as a continuous image of a compact set, and is closed as a compact subset of a Hausdorff space. But this projection is $f^{-1}(F)$, which completes the proof.
\end{proof}

\begin{proof}[Proof of Claim~\ref{claim:Phi_is_continuous}]
Let us first explain how to each $p$-tuple $(x_1,\ldots,x_p)=:x\in P_p\subset F_p(\mathbb R^2)$ we associate a unique set of Voronoi weights $w:=(w_1,\ldots, w_p)$ so that the corresponding weighted Voronoi partition with centers $x$ and weights $w$ splits $C$ into equal area parts $C_1(x,w),\ldots, C_p(x,w)$.  This is explained in~\cite[Section~2]{ahk2014} by referring to general properties of monotone transportation, but let us show existence and uniqueness to make our argument more self-contained.

The Voronoi cells, intersected with $C$, are defined as
\begin{equation}
\label{equation:voronoi}
C_i(x,w) = \left\{z\in C\ |\ |z - x_i|^2 - w_i \le |z - x_j|^2 - w_j\ \forall j\neq i\right\}.
\end{equation}
They obviously depend continuously on $x$ and $w$ whenever as long as they have non-empty interiors, and their areas depend continuously on $x$ and $w$ always. Consider the simplex 
\[
\Delta^{p-1}=\{(t_1,\ldots, t_p)\ |\ t_1+\dots+t_p=1,\ t_1,\ldots,t_p\ge 0\}
\]
and a map $\pi : \Delta^{p-1}\to \Delta^{p-1}$ defined by letting
\[
\pi(t)=\frac{1}{\area C}\left( \area C_1(x, w(t)), \ldots, \area C_p(x, w(t))\right),
\]
where $w(t)$ is defined by $w_i = -1/t_i$, weight $w_i=-\infty$ being allowed and implying $C_i(x,w)=\emptyset$.

The map $\pi : \Delta^{p-1}\to\Delta^{p-1}$ is continuous and has the property that $\pi(t)_i = 0$ when $t_i=0$ (and $w_i = - \infty$). So $\pi$ maps faces of $\Delta^{p-1}$ to themselves and is homotopic to the identity map by the linear interpolation, that is a homotopy in the class of PL maps of the pair $(\Delta,\partial\Delta)$ to itself. It follows that the degree of $\pi$ as a map of the pair $(\Delta^{p-1}, \partial\Delta^{p-1})$ to itself is $1$ and therefore $\pi$ is surjective. So we have $\pi(t) = (1/p,\ldots,1/p)$ for some $t$ and therefore have some nonzero weights that correspond to equipartition of the area.

Let us show the uniqueness of the weights up to adding the same number to all of the weights. Assume that two combinations of weights $w$ and $w'$ equipartition the area. By adding the number $\max\{w'_i-w_i\ |\ i=1,\ldots,p\}$ to all the components of $w$ we may assume that $w_i = w'_i$ for some $i$ and $w_j\ge w'_j$ for any $j$. From these inequalities it follows that for the given $i$ and any $j\neq i$
\[
|z - x_i|^2 - w_i \le |z - x_j|^2 - w_j\Rightarrow |z - x_i|^2 - w'_i \le |z - x_j|^2 - w'_j,
\]
that is $C_i(x,w)\subseteq C_i(x,w')$ and therefore $\area C_i(x,w)\le \area C_i(x,w')$. Since by our assumption the areas of parts actually equal $\area C/p$, this implies $C_i(x,w) = C_i(x,w')$. If $C_i(x,w)$ and $C_j(x,w)$ are adjacent (share a common edge in the interior of $C$) then the line separating them is given by the following equation in $z$,
\[
(z - x_i)^2 - w_i = (z - x_j)^2 - w_j\Leftrightarrow 2 z \cdot (x_j-x_i) + x_i^2 - x_j^2 = w_i - w_j.
\]
A strict inequality $w_j>w'_j$ would show that $C_i(x,w')$ is strictly larger than $C_i(x,w)$. Thus $w_j=w'_j$ for all cells adjacent to $C_i(x,w)$. Using the fact that the graph of the parts is connected under the adjacency relation, we see that actually $w=w'$, establishing the uniqueness of weights.

Note that the Voronoi weights in our problem can be assumed to be taken from a compact domain $\Omega_p$. The Voronoi centers $x_i$ are chosen from the union of the compact images of the projections $P_p\hookrightarrow F_p(\mathbb R^2)\to \mathbb R^2$ of the configuration space subpolyhedra to their component points in the plane. That is, the centers are chosen from a compact subset of the plane. The separating hyperplane must intersect the original convex body $K$ of Theorem~\ref{theorem:main}, since it is a hyperplane separating two non-empty parts of its iterated partition, and hence some $z$ in the equation of the separating hyperplane belongs to the body $K$. Hence the differences $w_i-w_j$ are all bounded. Since the weights of a weighted Voronoi partition may be shifted and normalized to satisfy $w_1+\dots + w_p = 0$, it is sufficient to choose them from a compact subset $\Omega_p\subset \mathbb R^p$, for example, a sufficiently large cube.

Because $\phi$ is continuous, it is sufficient to check that the map 
\[
\mathcal K_A\times P_p\to \underbrace{\mathcal K_{A/p}\times\dots \times \mathcal K_{A/p}}_p
\]
given by the formula
\[
(C,x)\to C_1(x)\times\ldots\times C_p(x)
\]
is continuous. Recall that $\mathcal K_A$ and $\mathcal K_{A/p}$ are Haudorff compacts by the Blaschke selection theorem and $P_p$ is compact by definition. So, by Theorem~\ref{theorem:closed-graph} it is sufficient to check that the graph of this map is compact. The graph is the Cartesian projection of a subset $X\subseteq \mathcal K_A\times P_p\times \Omega_p\times \underbrace{\mathcal K_{A/p}\times\dots \times \mathcal K_{A/p}}_p$ to the space $\mathcal K_A\times P_p\times \underbrace{\mathcal K_{A/p}\times\dots \times \mathcal K_{A/p}}_p$ and is therefore compact itself if $X$ is compact. 

It remains to show that $X$ is compact. Let us write down its explicit  definition:
\begin{multline*}
X = \{ (C, x, w, C_1, \ldots C_p)\in \mathcal K_A\times P_p\times \Omega_p\times \underbrace{\mathcal K_{A/p}\times\dots \times \mathcal K_{A/p}}_p\ | \\
C_i = \left\{z\in C\ |\ |z - x_i|^2 - w_i \le |z - x_j|^2 - w_j\ \forall j\neq i \right\} \},
\end{multline*}
noting that the factor $\underbrace{\mathcal K_{A/p}\times\dots \times \mathcal K_{A/p}}_p$ in the definition means the equality of areas 
\[
\area C_1 = \dots = \area C_p = A/p>0. 
\]
For positive areas of the $C_i$ the expression $C_i = \left\{z\in C\ |\ |z - x_i|^2 - w_i \le |z - x_j|^2 - w_j\ \forall j\neq i \right\}$ is continuous in $(C, x, w)$, hence the equations in the definition of $X$ define a closed set $X$. Since $\mathcal K_A\times P_p\times \Omega_p\times \underbrace{\mathcal K_{A/p}\times\dots \times \mathcal K_{A/p}}_p$ is compact, then its closed subset $X$ is compact.
\end{proof}
 
Fix a body $C$ and consider the restriction of $\Phi$ to it
\[
\Phi_C = \Phi|_{\{C\}\times P_p \times [-1,1]}.
\]
Note that $\Phi_C$ satisfies the following boundary conditions because $\phi$ is a nice multivalued function:
\begin{equation}
\label{equation:boundary_conditions}
\Phi_C(P_p \times \{-1\})\subset \mathbb R^p_- \quad\text{and}\quad \Phi_C(P_p \times \{1\})\subset \mathbb R^p_+.
\end{equation}
Here $\mathbb R^p_-=\{(x_1, \ldots, x_p):\forall i\: x_i < 0\}$, and $\mathbb R^p_+=\{(x_1, \ldots, x_p):\forall i \: x_i > 0\}$.

\begin{claim}
\label{claim-nonzero-cycle}
Let $T : P_p \times [-1,1]\to \mathbb R^p$ be a transverse to zero $G_p$-equivariant PL-map which satisfies the boundary conditions \eqref{equation:boundary_conditions}. Then the set $T^{-1}(0)$ is finite and has number of $G_p$-orbits counted with signs not divisible by $p$.
\end{claim}

\begin{proof}
Let us first present an instance of a map $T_0 : P_p\times [-1,1] \to \mathbb R^p$, which is transverse to zero, $G_p$-equivariant, and satisfies the boundary conditions \eqref{equation:boundary_conditions}, and for which the set $T_0^{-1}(0)$ consists of $G_p$-orbits whose number counted with signs is not divisible by $p$. In order to produce $T_0$, we may take the $\mathfrak S_p$-equivariant transverse to zero map 
\[
\tau_0 : P_p \to W_p = \{(x_1,\ldots,x_p):x_1+\ldots+x_p=0\}\subset \mathbb R^p
\]
from Lemma~\ref{lemma:bu}, such that $\tau_0^{-1}(0)$ has number of $G_p$-orbits counted with signs not divisible by $p$.
We augment $\tau_0$ to the transverse to zero map (assuming the coordinates of the image of $\tau_0$ are in the interval $(-1,1)$)
\[
T_0(x,y) = \tau_0(x) + \left(y,\ldots,y\right).
\]
Then $T_0^{-1}(0) = \tau_0^{-1}(0)\times \{0\}$ and this preimage still has number of $G_p$-orbits counted with signs not divisible by $p$.

Now the statement of the lemma follows from Lemma~\ref{lemma:equivariant-euler} applied to $T$ and $T_0$. The domain here is $P_p\times[-1, 1]$ which is a pseudomanifold modulo $p$ relative to $P_p\times \{-1,1\}$. Both $T$ and $T_0$ satisfy boundary conditions \eqref{equation:boundary_conditions} and the homotopy defined by $h(x,y,t)=(1-t)T_0(x,y)+tT(x,y)$ has the property required in Lemma~\ref{lemma:equivariant-euler} because $h(P_p\times \{-1\}\times [0,1])\subset \mathbb R_-^p$ and $h(P_p\times \{1\}\times [0,1])\subset \mathbb R_+^p$ imply $h(P_p\times \{-1,1\}\times [0,1])\not\ni 0$.
\end{proof}

\begin{claim}
\label{claim:one-parametric}
Let $\widetilde T : P_p \times [-1,1] \times [a,b]\to \mathbb R^p$ be a $G_p$-equivariant continuous map such that $\widetilde T_t:=\widetilde T|_{P_p \times [-1,1] \times \{t\}}$ satisfies the boundary conditions \eqref{equation:boundary_conditions} for all $t$. Then the projection of the set $\widetilde T^{-1}(0)$ to the rectangle $[-1,1] \times [a,b]$ separates the right side $\{1\} \times [a,b]$ from the left side $\{-1\} \times [a,b]$.
\end{claim}

\begin{proof}
Let us first assume that $\widetilde T$ is PL and transverse to zero with respect to a product triangulation of $P_p\times [-1,1]\times [a,b]$. The polyhedron $P_p\times [-1,1]\times [a,b]$ then is a pseudomanifold modulo $p$ relative to the union of its three subpolyhedra $(P_p\times\{-1, 1\})\times [a,b]$, $(P_p\times[-1, 1])\times \{a\}$, and $(P_p\times[-1, 1])\times \{b\}$. Thanks to the boundary conditions 
\[
\widetilde T(P_p \times \{-1\} \times [a,b]) \subset \mathbb R_-^p\quad\text{and}\quad\widetilde T(P_p \times \{1\} \times [a,b])\subset \mathbb R_+^p
\]
we have $\widetilde T(P_p\times\{-1, 1\}\times [a,b])\not\ni 0$ and so Lemma~\ref{lemma:equivariant-homotopy} applies to $\widetilde T$. By the lemma, the preimage of zero $Z = \widetilde T^{-1}(0)$ is a one-dimensional oriented graph with a free action of $G_p$ preserving the orientation. The quotient $Z/G_p$ is again an oriented graph.

Let $F:Z/G_p\to [-1,1] \times [a,b]$ be the projection map, its image coincides with the image of the projection $Z\to [-1,1] \times [a,b]$. By a property of product triangulations, $F$ is PL, that is, is affine on every edge of $Z/G_p$.

By Lemma~\ref{lemma:thom}, for all but a finite number of values $t\in[a,b]$ the map $\widetilde T_t$ is transverse to zero. The graph $Z$ has a finite number of vertices. So, there are $t_0\in(a,b)$ and a small $\delta > 0$ such that for all values $t\in [t_0-\delta,t_0+\delta]$ the map $\widetilde T_t$ is transverse to zero, $(P_{p}\times[-1,1]\times [t_0-\delta, t_0+\delta])\cap Z$ contains no vertex of $Z$ and consists of several oriented line segments connecting $P_{p}\times[-1,1]\times \{t_0-\delta\}$ to $P_{p}\times[-1,1] \{t_0+\delta\}$.

Then every point $(z,t_0)\in \widetilde T_{t_0}^{-1}(0)$ is the intersection of such a line segment with $P_{p}\times [-1, 1]\times \{t_0\}$. The sign of the point $(z, t_0)$ in the preimage of zero $\widetilde T_{t_0}^{-1}(0)$ is $+1$ if the coordinate $t$ increases along the line segments and $-1$ otherwise, by Lemma~\ref{lemma:equivariant-homotopy}(ii) applied to the restriction of $\widetilde T$ to $P_{p}\times [-1, 1]\times [t_0-\delta, t_0]$ (see Remark~\ref{remark:truncated-cylinder} for the explanation why the lemma applies). By Claim~\ref{claim-nonzero-cycle} we know that the sum of signs of $G_{p}$-orbits of points in $\widetilde T_{t_0}^{-1}(0)$ is not divisible by $p$. 

Passing to the quotient by the action of $G_{p}$, we see that  $(P_{p}/G_p\times [-1, 1]\times [t_0-\delta, t_0+\delta])\cap Z/G_p$ consists of several oriented line segments connecting $P_{p}/G_p\times [-1, 1]\times \{t_0-\delta\}$ to $P_{p}/G_p\times [-1, 1]\times \{t_0+\delta\}$. Moreover, the number of segments along which $t$ increases minus the number of segments along which $t$ decreases is not divisible by $p$. 

By Lemma~\ref{lemma:equivariant-homotopy}(ii,iii), the graph $Z/G_p$ is a PL $1$-cycle modulo $p$ relative to $P_p/G_p \times [-1,1] \times \{a,b\}$ and $F$ is affine on its edges. Thus the image of the cycle's projection $\zeta=F(Z/G_{p})$ is a PL $1$-cycle modulo $p$ in the rectangle $[-1,1] \times [a,b]$ relative to its top and bottom sides $[-1,1] \times \{a,b\}$, and disjoint with its left and right sides $\{-1,1\} \times [a,b]$ by the boundary conditions.

Since the projection $F$ preserves the $t$ coordinate, each oriented line segment in $(P_{p}/G_p\times [-1, 1]\times [t_0-\delta, t_0+\delta])\cap Z/G_p$ is mapped by $F$ to a line segment in $[-1, 1]\times [t_0-\delta, t_0+\delta]$ connecting $[-1, 1]\times \{t_0-\delta\}$ to $[-1, 1]\times \{t_0+\delta\}$. Only these line segments of $\zeta$ intersect the line $\{t=t_0\}$ because the remaining part of $\zeta$ lies outside $[-1, 1]\times [t_0-\delta, t_0+\delta]$. The number of such line segments along which $t$ increases minus the number of line segments along which $t$ decreases is not divisible by $p$. Summing up, we obtain that the intersection number $\zeta\cdot \{t=t_0\}$ is not divisible by $p$. 

Consider a curve $\gamma$ passing from the left $\{-1\} \times [a,b]$ to the right $\{1\} \times [a,b]$ of the rectangle $[-1,1] \times [a,b]$. If $\gamma$ does not intersect the support of $\zeta$ then the same is true for its perturbation not touching the top and the bottom of the rectangle, so we assume that $\gamma$ does not touch the top and the bottom of the rectangle. Then $\gamma$ is homologous to a horizontal line relative to the left and right sides of the rectangle. By Lemma~\ref{lemma:intersection-mod-p}(ii,iii), applied to the cycles $\gamma$, $\zeta$ and the subsets $A=\{-1,1\}\times [a,b]$ and $B=[-1,1]\times\{a,b\}$, we have that 
\[
\zeta\cdot\gamma=\zeta\cdot \{t=t_0\}\neq 0\mod p
\] 
and $\gamma$ must intersect the support of $\zeta$. 

We have proved the claim for a transverse to zero $\widetilde T$, now consider the general case of a continuous map. Consider a curve $\gamma$ passing from the left $\{-1\} \times [a,b]$ to the right $\{1\} \times [a,b]$ of the rectangle $[-1,1] \times [a,b]$. Let $\Gamma = F^{-1}(\gamma)\subset P_p \times [-1,1] \times [a,b] $ be the compact preimage of this curve. Assume contrary to the statement of the claim that $\widetilde T(\Gamma)\not\ni 0$ and therefore 
\[
\dist(\widetilde T(\Gamma), 0) > \varepsilon,\quad \dist(\widetilde T(P_p \times \{-1\} \times [a,b]), \partial\mathbb R_-^p) > \varepsilon,\quad \dist(\widetilde T(P_p \times \{1\} \times [a,b]), \partial\mathbb R_+^p) > \varepsilon
\] 
for some $\varepsilon>0$. 

By Lemma~\ref{lemma:equivariant-approximation} we may approximate $\widetilde T$ with precision $\varepsilon$ by 
a $G_p$-equivariant transverse to zero map $\widetilde T'$, PL with respect to sufficiently fine product triangulation of $P_p\times [-1,1]\times [a,b]$. Then 
\[
\widetilde T'(\Gamma)\not\ni 0,\quad \widetilde T'(P_p \times \{-1\} \times [a,b]) \subset \mathbb R_-^p,\quad \widetilde T'(P_p \times \{1\} \times [a,b])\subset \mathbb R_+^p.
\] 
But we have already proved that the projection of $\widetilde T'^{-1}(0)$ to the rectangle intersects $\gamma$, which is equivalent to $\widetilde T'^{-1}(\Gamma)\ni 0$, a contradiction.
\end{proof}

Denote by $Z$ the projection  of the set $S=\Phi^{-1}(0)\subset\mathcal K_A\times P_p\times [-1,1]$ to $\mathcal K_A\times [-1,1]$.

\begin{claim}
The set $Z$ separates the top $\mathcal K_A\times \{1\}$ from the bottom $\mathcal K_A \times \{-1\}$ of the cylinder $\mathcal K_A\times [-1,1]$.
\end{claim}

\begin{proof}
Assume that we have a continuous curve
\[
\gamma : [a,b] \to \mathcal K_A\times [-1,1]
\] 
passing from the bottom $\mathcal K_A\times \{-1\}$ to the top $\mathcal K_A\times \{1\}$ in the cylinder and parameterized by a segment $[a,b]$. Call its coordinates $C(s)$ and $y(s)$, the latter passing from $-1$ to $1$.

Define  
\[
\widetilde\Phi:P_p\times [-1,1]\times [a,b] \to \mathbb R^p \quad \widetilde\Phi(x, y, s) =\Phi(C(s), x, y).
\]

The curve $(y(s), s)$ passes from the left to the right side of the rectangle $[-1,1]\times [a,b]$. So, by Claim~\ref{claim:one-parametric} applied to $\widetilde\Phi$ it is intersected by the projection of $\widetilde\Phi^{-1}(0)$. Which means that there is $x\in P_p$ such that for some $s$
\[
\widetilde\Phi(x,y(s),s) = 0.
\]

By definition of $\widetilde\Phi(x, y, s) =\Phi(C(s), x, y)$ this effectively means that $(C(s),x,y(s))\in S$ and $(C(s), y(s))\in Z$. Hence $\gamma$ intersects $Z$.
 
\end{proof}

Let us show that the separation property of $Z\subset \mathcal K_A\times [-1,1]$ allows us to consider it as a graph of a multivalued function.

\begin{claim}
A closed set $Y\subset \mathcal K_A\times (-1,1)$ which separates the top $\mathcal K_A\times \{1\}$ from the bottom $\mathcal K_A \times \{-1\}$ is the graph of a nice multivalued function from $\mathcal K_A$ to $[-1,1]$.
\end{claim}
\begin{proof}
Take the distance to the set $Y$ function, $\dist(\cdot, Y)$,  under some metrization of $\mathcal K_A\times [-1,1]$, it is continuous and positive on the complement of $Y$. Since the top and the bottom of $\mathcal K_A\times [-1,1]$ belong to different connected components of the complement, we can flip the sign of this function on the bottom component to make it positive on the top and negative on the bottom. In effect, we obtain a function
\[
\psi : \mathcal K_A\times [-1,1] \to \mathbb R,
\]
satisfying the boundary condition sufficient to call its corresponding multivalued function with the graph $\{\psi(C,y)=0\}=Y$ nice. 
\end{proof}

By the previous two claims there exists a nice multivalued function $\psi$ with graph $Z$.
Let us check that it satisfies the statement of Lemma~\ref{lemma:function-to-function}.

 Whenever $\psi(C, y) = 0$ we have $(C,y)\in Z$ (since $Z$ is the graph of $\psi$). Since $Z$ is the projection of $S$, there is $x\in P_p$ such that $(C,x,y)\in S$. This triple, in turn, provides a partition of $C$ into $p$ convex bodies $C_1,\ldots, C_p$ satisfying 
\[
\phi(C_1, y) = \dots = \phi(C_p, y)
\]
by the definition of $S$ and $\Phi$.

This finishes the proof of Lemma~\ref{lemma:function-to-function}.

\section{Appendix: Remarks on the polyhedron $P_p$}

\subsection{Explanation of the construction by Blagojevi\'c and Ziegler}
\label{section:bz-explanations}

The anonymous referees (for both this paper and \cite{avvakumov2020}) asked for more explanations on the structure of the polyhedron $P_m\subset F_m(\mathbb R^2)$. We encourage the interested readers to read \cite{bz2014} (having in mind the case $d=2$ there for simplicity), but we also outline the main idea under the construction of $P_m$ there.

One notices that the ordered $m$-tuples of pairwise distinct points in $\mathbb R^2$, which is the definition of $F_m(\mathbb R^2)$, may be classified as follows. First, we classify the projection sets to the $x$ axis, those are multisets of labeled points. Basically, the projections form several subsets of labels, ordered left to right, each subset consisting of points with the same projection. Then for every such subset we also trace the order of $y$ coordinates of its points before the projection.

\begin{figure}[ht]
\center
\includegraphics[width=150mm]{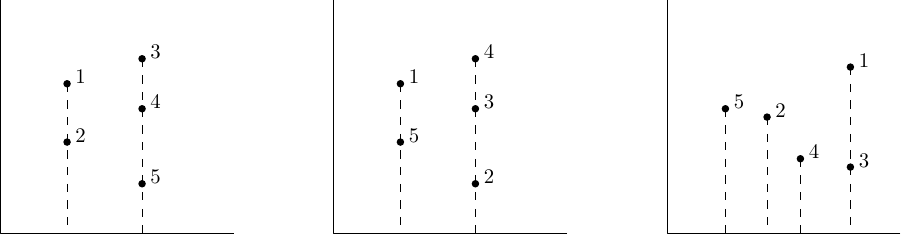}
\caption{The configurations of two $7$-dimensional unbounded cells and one $9$-dimensional unbounded cell of $F_5(\mathbb R^2)$, the 
$7$-cells belonging to the same $\mathfrak S_5$-orbit.}
\label{figure:pulling}
\end{figure}

Those combinatorial data define a decomposition of $F_m(\mathbb R^2)$ into convex cells, each given by a system of linear equations and strict linear inequalities. This is not a cellular decomposition in the usual sense because the cells are not closed, and closures of the cells are not compact. Note that the cells of the lowest dimension $m+1$ correspond to all $m$ points positioned on a vertical line. Hence there is a single $\mathfrak S_m$-orbit of these cells, $\mathfrak S_m C$, where $C$ is the cell of points satisfying $x_1=\dots=x_m$, $y_1<\dots<y_m$.

The construction of $P_m\subset F_m(\mathbb R^2)$ in \cite{bz2014} may be viewed as taking a $\mathfrak S_m$-invariant barycentric subdivision like the one we used in the proof of Lemma~\ref{lemma:thom}, considering sequences of cells whose closures form an inclusion-increasing sequence. This barycentric subdivision does not subdivide the whole $F_m(\mathbb R^2)$ but instead provides a polyhedron $P_m$ inside it. This procedure of passing to $P_m$ loses $m+1$ of the dimension because the lowest dimension of a cell was $m+1$. Finally one obtains the compact cells of $P_m$ of dimensions from $0$ to $m-1$ out of the initial unbounded cells of dimensions from $m+1$ to $2m$. 

Now we explain the properties of $P_m$ and the test map $f : P_m\to W_m$ stated in the beginning of Section~\ref{subsection:pk} and used in the proof of Lemma~\ref{lemma:bu}.

\begin{proof}[Construction of the test map $f$]
The test map $\widetilde f : F_m(\mathbb R^2)\to W_m$ corresponds to projecting the $m$ points to the line,
\[
((x_1,y_1),\ldots, (x_m,y_m)) \mapsto (x_1,\ldots, x_m),
\] 
and then sending them to the linear quotient $W_m=\mathbb R^m/\Delta$. The map $\widetilde f$ is just a surjective linear map and $\widetilde f^{-1}(0) = \mathfrak S_m C$. Hence its restriction $f = \widetilde f|_{P_m}$ has $f^{-1}(0)$ consisting of a single orbit of vertices $\mathfrak S_m v_0$ of $P_m$, the barycenters of the $\mathfrak S_m C$. 

This $f : P_m\to W_m$ is not transverse to zero in our definition, but a closer investigation shows that it satisfies the local homeomorphism assumptions of Lemma~\ref{lemma:local-degree} that we use in the proof of Lemma~\ref{lemma:bu} to produce a transverse to zero map $\tau_0$ from $f$. This is known from~\cite[Lemma~4.1]{bz2014}, but let us briefly explain this. Let $\Sigma$ be the star of $v_0$ in $P_m$. The vertices of $\Sigma$ correspond to the cells containing $C$ in their closures. The images of these cells under $f$ constitute a fan of convex cones in $\mathbb R^{2m}/C = W_m$ with apexes at the origin. Hence the barycentric subdivision of this fan is a PL ball $f(\Sigma)\subset \mathbb R^{2m}/C = W_m$ and $f$ maps $\Sigma$ to $f(\Sigma)$ PL homeomorphically. Hence $\Sigma$ may be chosen as $U_{v_0}$ in Lemma~\ref{lemma:local-degree} and $g\Sigma$ may be chosen as $U_{gv_0}$ for any $g\mathfrak S_m$.

Let us check the pseudomanifold modulo $p$ property of $P_m$ when $m=p^k$ with prime $p$. Every face of dimension $m-1$ of $P_m$ has precisely one vertex in $\mathfrak S_m v_0$, hence the stars $g\Sigma$, $g\in\mathfrak S_m$, of the vertices $gv_0$ cover the whole $P_m$. This covering by $g\Sigma$ can be extended to a ``cellular decomposition`` of $P_m$ (that we mention in Section~\ref{subsection:pk}), but the cells are PL balls rather than convex polytopes.
\end{proof}

\begin{proof}[Deduction of the pseudomanifold modulo $p$ property from Lemma~4.2 of \cite{bz2014}]
Choose an orientation of the star $\Sigma$ of $v_0$, then extend it to the whole $P_m$ by the requirement that $\mathfrak S_m$ flips the orientation with the sign of the permutation. In order to check the pseudomanifold modulo $p$ property, it remains to count the coefficients of codimension $1$ faces (or cells) in the boundary of $\sum_{g\in \mathfrak S_m} \sgn g\cdot g\Sigma$ viewed as an equivariant chain.

Lemma~4.2 of \cite{bz2014} asserts that the coboundary of an $\mathfrak S_m$-orbit of cells of $P_m$ of codimension $1$ (with a suitably chosen orientations related to the sign action of $\mathfrak S_m$) is divisible by $\binom{m}{j}$ for some $j=1,\ldots,m-1$, which is in turn divisible by $p$ when $m=p^k$ by Lucas' theorem \cite{lucas1878}. More precisely, by \cite[Lemma~4.2]{bz2014}, for a certain cell $T_j$ of codimension $1$ in $P_m$, the sum of coboundaries of its orbit satisfies
\[
\left\langle \sum_{g\in \mathfrak S_m} \sgn g \cdot d g T_j, \Sigma\right\rangle = 0\mod p.  
\]
Here $d$ is the coboundary operator, the natural product of (co)chains is defined on faces as $\langle F,G \rangle = 1$ when $F=G$ and $0$ otherwise. The orbits of these cells $T_j$, for $j=1,\ldots, m-1$, are all cells of codimension $1$ in the ``cellular decomposition'' of $P_m$.

Rewrite this to express in terms of the boundary operator $\partial$
\begin{multline*}
\sum_{g\in \mathfrak S_m} \sgn g \langle g dT_j, \Sigma\rangle =  \sum_{g\in \mathfrak S_m} \sgn g \langle d T_j, g^{-1}\Sigma\rangle = \\
= \sum_{g\in \mathfrak S_m} \sgn g \langle T_j, \partial g^{-1} \Sigma\rangle = \left\langle T_j, \partial \sum_{g\in \mathfrak S_m} \sgn g \cdot g^{-1} \Sigma\right\rangle = 0\mod p.  
\end{multline*}
Since all codimension $1$ cells of $P_m$ belong to orbits of such $T_j$ for $j=1,\ldots,m-1$, we obtain that $\sum_{g\in \mathfrak S_m} \sgn g\cdot g\Sigma_0$ has boundary divisible by $p$, which is precisely the pseudomanifold modulo $p$ property of $P_m$.
\end{proof}

\begin{proof}[A more direct explanation of the pseudomanifold modulo $p$ property]
Let us give a brief explanation of the above formulas and the origin of the cells $T_j$. In the cellular decomposition of $F_m(\mathbb R^2)$ (viewed as a subset of $\mathbb R^{2n}$) there is a unique orbit of the cell $C$ of the lowest dimension $m+1$, the $C$ being defined by
\begin{eqnarray*}
x_1 = \dots = x_m\\
y_1 < \dots < y_m.
\end{eqnarray*}
The cells of dimension $m+2$ belong to orbits of the cells $D_j$ for $j=1,\ldots, m-1$ defined by
\begin{eqnarray*}
x_1 = \dots = x_j < x_{j+1} = \dots = x_m\\
y_1 < \dots < y_j,\quad y_{j+1} < \dots < y_m.
\end{eqnarray*}
The boundary $\partial D_j$ then corresponds to a combination of $\binom{m}{j}$ elements of the form $\sgn g\cdot g C$, where $g$ is a permutation with subsequences $1,\ldots, j$ and $j+1,\ldots, m$ in the given order. Therefore in the cellular decomposition of $F_m(\mathbb R^2)$
\[
\sum_{g\in \mathfrak S_m} \sgn g \left\langle \partial D_j, g^{-1} C\right\rangle = \sum_{g\in \mathfrak S_m} \sgn g \left\langle g \partial D_j, C\right\rangle =
\binom{m}{j} = 0\mod p.
\]

Then one observes that the cell $\Sigma$ of $P_m$ identifies with a ball in the orthogonal complement to $C$ in the ambient $\mathbb R^{2n}$ space, since this is a barycentric subdivision of the partition of $F_m(\mathbb R^2)$ into convex cells. This generalized barycentric subdivision of the cells containing $C$ in their closures in fact corresponds to the more classical barycentric subdivision of the corresponding fan of cones in the linear quotient of $\mathbb R^{2n}$ by the linear span of $C$, triangulating a PL ball.

The $(m-2)$-dimensional cells $T_j$ of the ``cellular decomposition'' of $P_m$ are assembled from faces of $P_m$ that do not intersect the orbit $\mathfrak S_m v_0$. Each of such $(m-2)$-faces then intersects precisely one orbit of a vertex corresponding to a $(m+2)$-dimensional cell of $F_m(\mathbb R^2)$, so we may assume that $T_j$ is assembled from $(m-2)$-faces that do not intersect $\mathfrak S_m v_0$ and have one vertex, the barycenter of $D_j$. Similarly to $\Sigma$, $T_j$ is identified with a PL ball in the orthogonal complement to $D_j$ in $\mathbb R^{2n}$. 

Since the permutation group $\mathfrak S_m$ does not change the orientation of the ambient $\mathbb R^{2n}$, its action on the orientation of the cells in the orbit of $C$ corresponds to its action on the orientation of the cells in the orbit of $\Sigma$. Similarly, its action on the orientation of the cells in the orbit of $D_j$ corresponds to its action on the orientation of the cells in the orbit of $T_j$. The (co)boundary relations and sign calculations for the orbits of $\Sigma$ and the $T_j$ then correspond to those for the orbits of $C$ and $D_j$, because $\langle\partial \Sigma, g T_j\rangle = \langle C, \partial g D_j\rangle$.
\end{proof}

\subsection{Implicit construction of a useful polyheron in $F_p(\mathbb R^2)$}
\label{section:alt-explanations}

During the outlined above construction of $P_p$ one has to check the properties of $P_p$ that we used in our proof (see Section~\ref{section:bz} for the explicit list of the properties), and they are checked in \cite[Sections~3, 4]{bz2014}. At a more abstract level those properties may be guaranteed by a more general implicit consideration that we outline here. 

The proof of the Borsuk--Ulam-type theorem for $\mathfrak S_p$-equivariant maps $\Phi : F_p(\mathbb R^2)\to W_p$ both in \cite{ahk2014} and \cite{bz2014} establishes that such a map does not exist because the first cohomology obstruction $\xi\in H_{\mathfrak S_p}^{p-1}(F_p(\mathbb R^2); \pm \mathbb Z/p\mathbb Z)$ is nonzero. Here $\pm$ means the action of $\mathfrak S_p$ on the twisted coefficients by the permutation sign. This cohomology fact is 
\begin{itemize}
\item
Stated explicitly in \cite[Corollaries~4.5, 4.6]{bz2014}. 
\item
Implicitly follows from \cite{ahk2014}, where the proof of the Borsuk--Ulam-type theorem works in terms of locally finite homology that may produce cohomology using the intersection with the ordinary finite homology, thus resulting in establishing a nonzero cohomology obstruction.
\item
Follows from the spectral sequence calculation relating non-equivariant cohomology of $F_p(\mathbb R^2)$ to its $\mathfrak S_p$-equivariant cohomology, see for example \cite[Corollaries~3.5, 3.6]{cohen-taylor1993}. That calculation shows that rows from $1$ to $p-2$ in the spectral sequence cannot hit the Euler class of the representation $W_p$ in the bottom row of the spectral sequence, thus allowing it to remain nonzero in the equivariant cohomology of $F_p(\mathbb R^2)$. That is, $\xi$ is nonzero in $H^{p-1}(\mathfrak S_p; \pm \mathbb Z/p\mathbb Z)$ surviving on all stages of the spectral sequence.
\end{itemize}

Starting from the non-vanishing cohomology obstruction $\xi\in H_{\mathfrak S_p}^{p-1}(F_p(\mathbb R^2); \pm \mathbb Z/p\mathbb Z)$, obtained in one of the three ways, one may go in the opposite direction. Triangulate $F_p(\mathbb R^2)$ equivariantly (this is going to be an infinite triangulation of an open manifold lifted from its quotient $F_p(\mathbb R^2)/\mathfrak S_p$) and find a finite simplicial cycle $Q_p$ of this triangulation representing a homology class from $H^{\mathfrak S_p}_{p-1}(F_p(\mathbb R^2); \pm \mathbb Z/p\mathbb Z)$ on which the cohomology class $\xi$ evaluates to not divisible by $p$. This cycle $Q_p$ has the following properties:

\begin{itemize}
\item
$Q_p$ is supported on a finite union of faces of dimension $p-1$ and is invariant with respect to the action of the permutation group $\mathfrak S_p$ on $F_p(\mathbb R^2)$. 
\item
The $(p-1)$-dimensional cells of $Q_p$ are oriented and assigned coefficients in $\mathbb Z/p\mathbb Z$ so that $\mathfrak S_p$ acts on these orientations and coefficients by the sign of the permutation. This rephrases the property that $Q_p$ is an equivariant cycle with coefficients in $\pm\mathbb Z/p\mathbb Z$.
\item
The homological boundary of $Q_p$ is divisible by $p$ because $Q_p$ is a cycle.
\item
For any transverse to zero $\mathfrak S_p$-equivariant map $\Phi : Q_p\to W_p$ the point set $\Phi^{-1}(0)$ consists of a number of $\mathfrak S_p$-orbits, and this number counted with coefficients of the $(p-1)$-dimensional cells of $Q_p$ (using the orientation of $Q_p$ and $W_p$ similar to Definition~\ref{definition:point-sign})) is not divisible by $p$. This rephrases the fact that the obstruction class $\xi$ is not divisible by $p$ on $Q_p$. 
\end{itemize}

These properties are sufficient to use the support of $Q_p$ in place of $P_p$ and the cycle $Q_p$ with orientations and multiplicities in place of the pseudomanifold $P_q$ with orientations of the top-dimensional faces. For example, Lemmas~\ref{lemma:equivariant-homotopy} and \ref{lemma:equivariant-euler} only need to count the signs and edges of the oriented graph with multiplicities. Eventually, the following analogue of Lemma~\ref{lemma:bu} holds true.

\begin{lemma}
\label{lemma:bu-q}
For a prime $p$ $G_p=\mathfrak S_p$ for $p=2$ or the even permutation subgroup of $\mathfrak S_p$ for odd prime $p$. Then

i$)$ There exists a map $\tau_0 : Q_p\to W_p$ that is transverse to zero and has number of $G_p$-orbits of points in $\tau_0^{-1}(0)$ counted with signs not divisible by $p$. 

ii$)$ Any transverse to zero $G_p$-equivariant map $\tau : Q_p\to W_p$ has number of $G_p$-orbits of points in $\tau^{-1}(0)$ counted with signs not divisible by $p$.

iii$)$ Any continuous $G_p$-equivariant map $\tau : Q_p\to W_p$ contains $0$ in its image.
\end{lemma}

Similarly to our main approach, passing to the subgroup $G_p\subseteq \mathfrak S_p$ here allows to avoid using twisted coefficients. It is justified by the fact that the cohomology restriction map is injective on the modulo $p$ equivariant cohomology since the composition of the restriction and the transfer in the opposite direction is the multiplication by $|\mathfrak S_p/G_p|\neq 0\mod p$. 

Alternatively, one may note (as in the proof of a more general \cite[Lemma~5]{kar2009conf}) that the cohomology calculation in \cite[Corollaries~3.5, 3.6]{cohen-taylor1993} can be applied directly to the cyclic group $Z_p\subseteq \mathfrak S_p$ to observe that $H^i(F_p(\mathbb R^2); \mathbb Z/p\mathbb Z)$ (the coefficients are not twisted in this case) is a direct sum of free $\mathbb Z/p\mathbb Z[Z_p]$-modules for $i=1,\ldots, p-2$, that prevents nonzero maps of the corresponding rows of the spectral sequence to the bottom row.

\section{Appendix: A weaker higher-dimensional result}
\label{section:higher-pos}

Now we are going to consider the case when we work in $\mathbb R^d$, have $d-1$ measures $\mu_1,\ldots, \mu_{d-1}$ in a convex body $K$ and want to partition $K$ into $m$ convex parts of equal $\mu_j$ measure (for every $j$) and equal surface area. As with the perimeter, the ``surface area'' may be any continuous function of a convex body in $\mathbb R^d$. 

In Appendix \ref{section:higher-neg} we explain why our approach is not suitable when we want to equalize two arbitrary functions and $d-2$ measures of parts, which is why we only dare to handle one arbitrary function here. Let us state the result:

\begin{theorem}
\label{theorem:main-d}
Let $d\ge 2$. Assume $d-1$ finite nonzero Borel measures $\mu_1,\ldots,\mu_{d-1}$ with non-negative density $($that is absolutely continuous with respect to the Lebesgue measure$)$ are given in a convex body $K\subset\mathbb R^d$ and $f$ is a continuous function of a convex body. If $m\ge 2$ is an integer then it is possible to partition $K$ into $m$ convex parts $V_1,\ldots, V_m$ so that for every $i$
\[
\mu_i(V_1) = \mu_i(V_2) = \dots = \mu_i(V_m)
\]
and 
\[
f(V_1) = f(V_2) = \dots = f(V_m).
\]
\end{theorem}

In \cite{soberon2012} a similar result was proved, in the case of $d$ measures and no arbitrary function. In terms of the previous section this is explained as follows: In the induction step we equalize $d$ measures in $p_1$-tuples of parts of the bottom level of the hierarchical partition, but we do not need to work with ``multivalued functions'' because the measures are additive and once we equalize the measures we know the single common value. 

As for the proof, our $d$-dimensional theorem follows from the following analogue of Lemma \ref{lemma:function-to-function}. Let $\mathcal K^d_M$ be the space of $d$-dimensional convex bodies contained in a fixed convex body and having $\mu_1$-measure at least $M$.

\begin{lemma}
\label{lemma:function-to-function-d}
Assume $\phi$ is a nice multivalued function of $\mathcal K^d_{M/p}$, $p$ is a prime, $\mu_1,\ldots,\mu_{d-1}$ are locally finite Borel measures with non-negative density $($that is absolutely continuous with respect to the Lebesgue measure$)$.

Then there exists another nice multivalued function $\psi$ of $\mathcal K^d_M$ such that whenever $C\in\mathcal K^d_M$ satisfies 
\[
\psi(C, y) = 0
\]
then there exists a partition $C = C_1\cup \dots \cup C_p$ into convex bodies such that for every $i=1,\ldots, d-1$,
\[
\mu_i(C_1) = \dots = \mu_i(C_p)
\] 
and 
\begin{equation}
\label{equation:equalized-p-d}
\phi(C_1, y) = \dots = \phi(C_p, y) = 0.
\end{equation}
\end{lemma}

The proof follows by considering the more general $(d-1)(p-1)$-dimensional pseudomanifolds modulo a prime $p$, $P_{p;d}\subset F_p(\mathbb R^d)$, introduced in \cite{bz2014}, with the group of symmetry $G_p$ as in the previous section. The map $\Phi_C : P_{p;d}\times [-1,1] \to \mathbb R^{(d-1)p}$ is then built from a configuration $x\in P_{p;d}$ of $p$ points in $\mathbb R^d$, considered as Voronoi centers. The measure $\mu_{d-1}$ is equalized by finding appropriate Voronoi weights and establishing a partition $C=C_1\cup\dots\cup C_p$.

The functions $\mu_i(C_j) - \frac{1}{p} \mu_i(C)$ ($i=1,\ldots, d-2$) and $\phi(C_i, y)$, $y\in [-1,1]$, then constitute the coordinates of $\Phi_C : P_{p;d}\times [-1,1] \to \mathbb R^{(d-1)p}$. Whenever such $\Phi_C$ is transverse to zero, its preimage of zero (where the measures $\mu_1,\ldots,\mu_{d-2}$ are equalized in the common sense and $\phi$ is equalized as a multivalued function) is has number of points counted with signs not divisible by $p$. This is a version of Claim \ref{claim-nonzero-cycle} in this more general situation.

The rest of the proof of Lemma \ref{lemma:function-to-function-d} is essentially the same as the proof of Lemma \ref{lemma:function-to-function}. Theorem \ref{theorem:main-d} follows as in the two-dimensional case, the measures are equalized since on every prime number stage the partition is a partition into parts of equal measures, the function $\phi$ is equalized as guaranteed by the lemma.
\begin{remark}
Of course, we were trying to find a generalization of this argument in order, for example, to equalize two arbitrary functions of the convex parts in $\mathbb R^3$ together with their volumes. A crucial obstacle, in our opinion, is that when we make an induction step and consider a ``subfunction'' of a multivalued function with a separation argument, then the procedure of restoring the subpartition (of a part in the hierarchy) corresponding to the chosen common value of this equalized function of the subpartition is not continuous. In particular the other function we want to equalize may not depend continuously (or be a nice multivalued function) on the first one after this choice.
\end{remark}

\section{Appendix: Difficulty of equalizing two arbitrary functions}
\label{section:higher-neg}

In this section we point out some essential difficulties in the attempt to generalize our technique to the case when we need to equalize at least two arbitrary continuous functions of convex parts. We thank Sergey Melikhov for sharing with us his ideas that developed into the argument of this section.

Assume that we have a convex body $K\subset\mathbb R^3$ and want to partition it into $m=2p^s$ ($p$ is an odd prime) convex parts with equal volumes, and equal values of two other functions $F_1, F_2$ of the parts continuous in Hausdorff metric. We would naturally start by partitioning $K$ into two parts of equal volume; such partitions are parametrized by the normal of the oriented partitioning plane, that is by the sphere $S^2$. Then in the part of $K$ the normal vector points to, we would apply the Blagojevi\'c--Ziegler argument with a polyhedral configuration space for $m=p^s$ to have a number of solutions for this half of the problem not divisible by $p$. Looking at the possible pairs of common values of $F_1,F_2$ we would obtain, as in the proof of the  main result of this paper, a multivalued function $S^2\to \mathbb R^2$, whose graph in $S^2\times \mathbb R^2$, under certain genericity assumptions, could be viewed as a $2$-dimensional cycle modulo $p$, which we denote by $Z$, homologous modulo $p$ to $k[S^2\times \{(0,0)\}]$ for some $k\neq 0\mod p$.

The problem would be solved this way if we could prove that under the antipodal map $\sigma : S^2\to S^2$, extended to $S^2\times \mathbb R^2$ by the trivial action of $\sigma$ on $\mathbb R^2$, some point of the support of $Z$ would go to some other point of the support of $Z$. But below we build an example of a modulo $p$ cycle $Z$ that satisfies all the assumptions that we know it must satisfy in the problem, but has disjoint $Z$ and $\sigma(Z)$.

Let us build $Z$ inside $S^2\times D$, where $D$ is the unit disk in the plane $\mathbb R^2$. Let us split $S^2$ by its equator $S^1$ into closed hemispheres $D_+$ and $D_-$. Start by building the part of $Z$ that lies over $S^1$: Let $L$ be the graph of 
\[
z\mapsto z^n,
\]
where we identify $D$ with the unit disk in the complex plane and $S^1$ with the unit complex numbers. For odd $n$ the circles $L$ and $\sigma L$ do not intersect and their linking number (if we consider the solid torus $S^1\times D$ lying standardly in $\mathbb R^3$) is $\lk (L, \sigma L) = n$, since for odd $n$ the circle $\sigma L$ is the graph of
\[
z\mapsto - z^n,
\]
and the linking number of two circles, close to each other, equals the winding number of their difference vector when we pass along the circles.

Letting this $n$ be equal to the prime number $p$ from the formula $m=2p^s$, we thus have that $L$ and $\sigma L$ are non-linked $1$-dimensional modulo $p$ cycles. Now we pass from the torus $S^1\times D$ to the topological $4$-dimensional ball $B^4 = D_+\times D$. The torus $S^1\times D$ is a part of its boundary $S^3=\partial B^4$ and the cycles $L$ and $\sigma L$ are non-linked modulo $p$ cycles in $S^3$, since the torus embeds into $S^3$ without a twist. It follows that we may choose two $2$-dimensional modulo $p$ cycles in $B^4$ relative to $S^3$, $M$ and $N$, so that $\partial M = L$, $\partial N = \sigma L$, and the supports of $M$ and $N$ are disjoint. 

Indeed, choose $M$ as any topologically embedded disk in $B^4$, whose boundary maps homeomorphically to the circle $L$. By Alexander duality for the pair $(B^4, S^3)$, we have
\[
H_1(B^4\setminus M; \mathbb Z/p\mathbb Z) = H^2(M, L; \mathbb Z/p\mathbb Z) = \mathbb Z/p\mathbb Z.
\]
Hence the homology class $[\sigma(L)]\in H_1(B^4\setminus M; \mathbb Z/p\mathbb Z)$ is fully determined by an element of $\mathbb Z/p\mathbb Z$, which is in fact the linking number $\lk(L, \sigma L)$. Having this linking number $0$ modulo $p$, we may conclude that $\sigma L$ is a modulo $p$ boundary of a $2$-dimensional modulo $p$ chain $N$ in $B^4\setminus M$. The chains $N$ and $M$ thus have disjoint supports.

Now we pass to $S^2\times D$ from $D_+\times D$ and take the $2$-dimensional modulo $p$ cycle $Z = M - \sigma(N)$, this is indeed a cycle, since 
\[
\partial Z = \partial M - \sigma(\partial N) = L - \sigma(\sigma(L)) = 0.
\]
From the construction of $M$ and $N$ we may conclude that $Z$ and $\sigma Z$ are disjoint. At the same time, $Z$ is homologous modulo $p$ to $[S^2\times \{(0,0)\}]$, which is equivalent to saying that it intersects $\{x\}\times D$, for generic $x\in S^2$, $1$ modulo $p$ number of times, counted with signs. The last claim is evidently true for $x\in S^1$, where 
\[
(\{x\}\times D) \cap Z = (\{x\}\times D)\cap L.
\]
Our construction of $M$ and $N$ allows them to have collars near $S^3\subset B^4$ that allows us to keep the uniqueness of such an intersection for $x$ in a neighborhood of $S^1$ in $S^2$. If we want $Z$ to be homologous to a multiple $k[S^2\times \{(0,0)\}]$ modulo $p$ then we may just repeat this construction in $k$ smaller disks $D_1,\ldots,D_s$ embedded in $D$ and take the sum of the obtained cycles.

Thus we have checked that $Z$ has the properties that a graph of the multivalued function from our attempted proof must have, but does not allow to make the final step of the proof.

\begin{remark}
Using several circles $L_i$, given by $z\mapsto c_i + \epsilon z^{n_i}$ for different $c_i\in D$, odd integers $n_i$, and sufficiently small $\epsilon>0$, it is possible to replace $L$ in the above argument with an algebraic combination $L' = \sum_i L_i$, such that 
\[
\lk (L', \sigma L') = \sum_i n_i = 0
\] 
as an integer. We may also make $L'$ modulo $p$ (but not integrally!) homologous to $k[S^1\times \{(0,0)\}]$, by choosing the number of the $L_i$ to equal $k$ modulo $p$. Then we choose $M$ as an oriented surface in $B^4 = D_+\times D$ with boundary $L'$, $N$ as a integral chain in $B^4\setminus M$ with boundary $\sigma(L')$. The integral chain $Z = M - \sigma(N)$ then becomes an integral cycle, modulo $p$ (but not integrally!) equivalent to $k[S^2\times \{(0,0)\}]$. And $Z$ is disjoint from $\sigma(Z)$, that is the Borsuk--Ulam theorem cannot be generalized to the corresponding multivalued map $S^2\to\mathbb R^2$.
\end{remark}

\bibliography{../Bib/karasev}

\bibliographystyle{abbrv}
\end{document}